\documentclass[11pt]{amsart}
\usepackage{amssymb,setspace}
\usepackage{multirow}

\newtheorem{theorem}{Theorem}[section]
\newtheorem{prop}[theorem]{Proposition}
\newtheorem{lemma}[theorem]{Lemma}
\newtheorem{corollary}[theorem]{Corollary}
\newtheorem{conj}[theorem]{Conjecture}
\newtheorem{prob}[theorem]{Problem}

\newcounter{Examplecount}
\newenvironment{example}[1][Example \arabic{Examplecount}.]{\vskip3pt \begin{trivlist}
\item[\hskip \labelsep {\bfseries #1}]}{\end{trivlist}\vskip2pt\stepcounter{Examplecount}}

\newcommand\beq{\begin{equation}}
\newcommand\eeq{\end{equation}}
\newcommand\bce{\begin{center}}
\newcommand\ece{\end{center}}
\newcommand\bea{\begin{eqnarray}}
\newcommand\eea{\end{eqnarray}}
\newcommand\bean{\begin{eqnarray*}}
\newcommand\eean{\end{eqnarray*}}
\newcommand\bmt{\begin{multline*}}
\newcommand\emt{\end{multline*}}
\newcommand\ben{\begin{enumerate}}
\newcommand\een{\end{enumerate}}
\newcommand\bit{\begin{itemize}}
\newcommand\eit{\end{itemize}}
\newcommand\brr{\begin{array}}
\newcommand\err{\end{array}}
\newcommand\bt{\begin{tabular}}
\newcommand\et{\end{tabular}}

\newcommand\bs{\bigskip}
\newcommand\ms{\medskip}

\renewcommand\S{\mathcal S}
\newcommand\C{\mathcal C}
\newcommand\wh{\widehat}
\newcommand\wt{\widetilde}
\newcommand\T{{\mathcal T}}
\newcommand\U{{\mathcal U}}
\newcommand\bij{\varphi}
\newcommand\inv{\psi}
\newcommand\biju{\phi}
\newcommand\eps{\varepsilon}
\newcommand\sigm{\wt\sigma}
\newcommand\pim{\wt\pi}

\author{Sergi Elizalde}
\title{Descent sets of cyclic permutations}
\address{Department of Mathematics, Dartmouth College, Hanover, NH 03755-3551}

\begin{document}

\maketitle

\begin{abstract}
We present a bijection between cyclic permutations of $\{1,2,\dots,n+1\}$ and permutations of $\{1,2,\dots,n\}$
that preserves the descent set of the first $n$ entries and the set of weak excedances.
This non-trivial bijection involves a Foata-like transformation on the cyclic notation
of the permutation, followed by certain conjugations.
We also give an alternate derivation of the consequent result about the equidistribution of descent sets using work of Gessel and Reutenauer.
Finally, we prove a conjecture of the author in [{\it SIAM J. Discrete Math.} 23 (2009), 765--786] and a conjecture of Eriksen, Freij and W\"astlund.
\end{abstract}


\section{Introduction}\label{sec:intro}

\subsection{Permutations, cycles, and descents}

Let $[n]=\{1,2,\dots,n\}$, and let $\S_n$ denote the set of permutations of $[n]$. We will use both the one-line notation of $\pi\in\S_n$
as $\pi=\pi(1)\pi(2)\dots\pi(n)$ and its decomposition as a product
of cycles of the form $(i,\pi(i),\pi^2(i),\dots,\pi^{k-1}(i))$ with $\pi^k(i)=i$.
For example, $\pi=2517364=(1,2,5,3)(4,7)(6)$. Sometimes it will be convenient to write
each cycle starting with its largest element and order cycles by increasing first element, e.g., $\pi=(5,3,1,2)(6)(7,4)$.

We denote by $\C_n$ the set of permutations in $\S_n$ that consist of one single cycle of length~$n$. We call these {\em cyclic permutations} or {\em $n$-cycles}.
For example, $$\C_3=\{(1,2,3),(1,3,2)\}=\{231,312\}.$$ It is clear that $|\C_n|=(n-1)!$.

Given $\pi\in\S_n$, let $D(\pi)$
denote the {\em descent set} of $\pi$, that is,
$$D(\pi)=\{i\,:\,1\le i\le n-1,\ \pi(i)>\pi(i+1)\}.$$
The descent set can be defined for any sequence of integers $a=a_1 a_2 \dots a_n$
by letting $D(a)=\{i\,:\,1\le i\le n-1,\ a_i>a_{i+1}\}$. We denote by $E(\pi)$ the set of {\em weak excedances} of $\pi$, that is,
$$E(\pi)=\{i\,:\,1\le i\le n,\ \pi(i)\ge i\}.$$

The main result of this paper, which we present in Section~\ref{sec:main}, is a bijection $\bij$
between $\C_{n+1}$ and $\S_n$ with the property that for every $(n+1)$-cycle $\pi$,
$$D(\pi(1)\pi(2)\dots\pi(n))=D(\bij(\pi)).$$
We give a relatively natural description of the map $\bij$. However, the proof that it is a bijection with the desired property is far from trivial, and is done in Section~\ref{sec:proof}.
We also show that $\bij$ preserves the set of weak excedances, namely $E(\pi)=E(\bij(\pi))$.

Let us introduce some notation. For $\pi\in\S_n$, let $\wh\pi$ be the permutation defined by $\wh\pi(i)=n+1-\pi(n+1-i)$ for $1\le i\le n$.
The cycle form of $\wh\pi$ can be obtained by replacing each entry $j$ with $n+1-j$ in the cycle form of $\pi$.
For $1\le i\le n-1$, we have that $i\in D(\wh\pi)$ if and only if $n+1-i\notin D(\pi)$.

We will write $I=\{i_1,i_2,\dots,i_k\}_<$ to indicate that the elements of $I$ are listed in increasing order.
Subsets $I\subseteq[n-1]$ are in bijective correspondence with compositions of $n$ via $\{i_1,i_2,\dots,i_k\}_<\mapsto (i_1,i_2-i_1,\dots,i_k-i_{k-1},n-i_k)$. The partition of $n$ obtained by listing the parts of this composition
in non-increasing order is called the associated partition of $I$. For example, for $n=13$ and $I=\{3,5,8,12\}$, the associated partition is $(4,3,3,2,1)$.

\subsection{Related work}

Following the notation from~\cite{Eli}, let $\T^0_n$ be the set whose elements are $n$-cycles in one-line notation in which one entry has been replaced with $0$.
For example, $\T^0_3=\{031,201,230,012,302,310\}$. Since there are $n$ ways to choose what entry to replace, and the value of the replaced entry can be recovered by looking at the other entries,
it is clear that $|\T^0_n|=n!$. Note that if the $0$ in $\tau\in\T^0_n$ is in position $i$, then $i-1\in D(\tau)$ (if $i>1$) and $i\notin D(\tau)$.
It was conjectured in~\cite{Eli} that descent sets in $\T^0_n$ behave like descent sets in $\S_n$:

\begin{conj}[\cite{Eli}]\label{conj:Eli} For every $n$ and any $I\subseteq[n-1]$,
$$|\{\tau\in\T^0_n\,:\,D(\tau)=I\}|=|\{\sigma\in\S_n\,:\,D(\sigma)=I\}|.$$
\end{conj}

In Section~\ref{sec:consequences} we prove this conjecture as Corollary~\ref{cor:Elishift}, along with other consequences of our main bijection.

There is some work in the literature relating the cycle structure of a permutation with its descent set.
Gessel and Reutenauer~\cite{GR} showed that the number of permutations with given cycle structure and descent set can be expressed as a product of certain characters of the symmetric group.
They also gave a statistic-preserving bijection between words and multisets of necklaces. In Section~\ref{sec:related} we discuss how their work relates to ours, and how their methods can be used to prove some of our results non-bijectively.

More recently, Eriksen, Freij and W\"astlund~\cite{EFW} studied descent sets of derangements. Recall that derangements are permutations with no fixed points.
In~\cite[Problem~9.3]{EFW}, the authors pose the following question:

\begin{prob}[\cite{EFW}]\label{prob:EFW} For any two subsets $I,J\subseteq[n-1]$ with the same associated partition, give a bijection between derangements of $[n]$ whose descent
set is contained in $I$ and derangements of $[n]$ whose descent set is contained in $J$.
\end{prob}

At the end of Section~\ref{sec:related} we solve this problem by giving a bijection based on the work of Gessel and Reutenauer~\cite{GR}.

\section{The main result}\label{sec:main}

\begin{theorem}\label{thm:bij} For every $n$ there is a bijection $\bij:\C_{n+1}\rightarrow\S_n$ such that if $\pi\in\C_{n+1}$ and $\sigma=\bij(\pi)$, then
$$D(\pi)\cap[n-1]=D(\sigma).$$
\end{theorem}

In this section we define the map $\bij:\C_{n+1}\rightarrow\S_n$ and give some examples. Next we describe a map $\inv:\S_n\rightarrow\C_{n+1}$. In Section~\ref{sec:proof} we will
prove that $\inv=\bij^{-1}$ and that $\bij$ preserves the descent set of the first $n$ entries.

\subsection{The map $\bij$}

Given $\pi\in\C_{n+1}$, write it in cycle form with $n+1$ at the end, i.e., $\pi=(t_1,t_2,\dots,t_n,n+1)$. Let $t_1=t_{i_1}<t_{i_2}<\dots<t_{i_r}<t_{i_{r+1}}=n+1$ be the left-to-right maxima of the
sequence $t_1,t_2,\dots,t_n,n+1$. Let $$\sigm=(t_{1},t_2,\dots,t_{i_2-1})(t_{i_2},t_{i_2+1},\dots,t_{i_3-1})\cdots(t_{i_r},t_{i_r+1},\dots,t_n).$$ To simplify notation, let $a_j=t_{i_j}$ and $b_j=t_{i_{j+1}-1}$ for $1\le j\le r$,
so \beq\label{eq:sigm}\sigm=(a_1,\dots,b_1)(a_2,\dots,b_2)\cdots(a_r,\dots,b_r).\eeq
To obtain $\bij(\pi)$ we will make some changes in $\sigm$, which we describe next. Each change consists of switching two entries in the cycle form of $\sigma$ given in~(\ref{eq:sigm}),
so the cycle type is preserved during the algorithm.
With some abuse of notation, we also denote by $\sigm$ the permutation obtained after each switch, and we write its cycle form as in~(\ref{eq:sigm}) with only the switched entries moved.
We denote by $\Gamma_i$ the $i$-th cycle of $\sigm$, with the cycles written from left to right as in~(\ref{eq:sigm}).
The terms {\em left}, {\em right}, {\em first} (or {\em leftmost}) and {\em last} (or {\em rightmost}) will always assume that the entries within each cycle are also written in this order.
Whenever we have two adjacent elements $s$ and $t$ in a cycle, with $s$ immediately to the left of $t$,
we will say that $s$ {\em precedes} $t$.
For $1\le x,y\le n$, let $P(x,y)$ be the condition
$$\pi(x)>\pi(y) \ \textrm{ and }\  \sigm(x)<\sigm(y).$$
(For $x$ or $y$ outside of these bounds, $P(x,y)$ is defined to be false.)\ms

Repeat the following steps for $i=1,2,\dots,r-1$:
\bit
\item Let $z$ be the rightmost entry of $\Gamma_i$.
If $P(z,z+1)$ or $P(z,z-1)$ holds, let $\eps\in\{-1,1\}$ be such that $P(z,z+\eps)$ holds and $\sigm(z+\eps)$ is largest.
\item Repeat for as long as $P(z,z+\eps)$ holds:
\ben\renewcommand{\labelenumi}{\Roman{enumi}.}
\item\label{it:a} Switch $z$ and $z+\eps$ in the cycle form of $\sigm$.
\item\label{it:b} If the last switch did not involve the leftmost entry of $\Gamma_i$, let $x$ and $y$ be the elements preceding the switched entries.
If $|x-y|=1$, switch $x$ and $y$ in the cycle form of $\sigm$, and repeat step~II.
\item Let $z:=z+\eps$ (the new rightmost entry of $\Gamma_i$).
\een
\eit

Define $\bij(\pi)=\sigm$.

\begin{table}[htb]
$$\begin{array}{|c|c|c|c|c|c|c|}
\hline
\pi\in\C_5 & \sigma=\bij(\pi)\in\S_4 &  D(\pi)\cap[3]=D(\sigma) \\[2pt] \hline
(1,2,3,4,5) = 23451  & 1 2 3 4     & \emptyset \\ \hline
(2,1,3,4,5) = 31452  & 2 1 3 4     & \multirow{3}{*}{\{1\}}  \\
(3,2,1,4,5) = 41253  & 3 1 2 4     &   \\
(4,3,2,1,5) = 51234  & 4 1 2 3     &   \\ \hline
(1,3,2,4,5) = 34251  & 1 3 2 4     & \multirow{5}{*}{\{2\}}  \\
(1,4,3,2,5) = 45231  & 1 4 2 3     &  \\
(3,1,2,4,5) = 24153  & 2 3 1 4     &  \\
(3,1,4,2,5) = 45123  & 3 4 1 2     &  \\
(4,3,1,2,5) = 25134  & 2 4 1 3     &  \\ \hline
(1,2,4,3,5) = 24531  & 1 2 4 3     & \multirow{3}{*}{\{3\}}  \\
(2,4,1,3,5) = 34512  & 1 3 4 2     &  \\
(4,1,2,3,5) = 23514  & 2 3 4 1     &  \\ \hline
(2,3,1,4,5) = 43152  & 3 2 1 4     & \multirow{3}{*}{\{1,2\}}\\
(2,4,3,1,5) = 54132  & 4 2 1 3     & \\
(4,2,3,1,5) = 53124  & 4 3 1 2     & \\ \hline
(1,4,2,3,5) = 43521  & 3 2 4 1     & \multirow{5}{*}{\{1,3\}}\\
(2,1,4,3,5) = 41532  & 2 1 4 3     & \\
(2,3,4,1,5) = 53412  & 4 2 3 1     & \\
(3,4,2,1,5) = 51423  & 4 1 3 2     &  \\
(4,2,1,3,5) = 31524  & 3 1 4 2     &  \\ \hline
(1,3,4,2,5) = 35421  & 1 4 3 2     &  \multirow{3}{*}{\{2,3\}}\\
(3,4,1,2,5) = 25413  & 2 4 3 1     &   \\
(4,1,3,2,5) = 35214  & 3 4 2 1     &  \\ \hline
(3,2,4,1,5) = 54213  & 4 3 2 1     & \{1,2,3\}\\
\hline
\end{array}$$
\caption{\label{tab:n4} The images by $\bij$ of all elements in $\C_5$.}
\end{table}

\begin{example}
Let $$\pi=(11,4,10,1,7,16,9,3,5,12,20,2,6,14,18,8,13,19,15,17,21)\in\C_{21}.$$
Finding the left-to-right maxima of the sequence, we get $$\sigm=(11,4,10,1,7)(16,9,3,5,12)(20,2,6,14,18,8,13,19,15,17).$$

Now we look at the first cycle, so $z=b_1=7$. Both $P(7,6)$ and $P(7,8)$ hold, but $\sigm(6)=14>13=\sigm(8)$, so $\eps=-1$.
Switching $7$ and $6$ we get $$\sigm=(11,4,10,1,{\bf 6})(16,9,3,5,12)(20,2,{\bf 7},14,18,8,13,19,15,17).$$
The entries preceding the switched ones are $1$ and $2$ so we switch them too: $$\sigm=(11,4,10,{\bf 2},6)(16,9,3,5,12)(20,{\bf 1},7,14,18,8,13,19,15,17).$$
Now $z=6$, and since $P(6,5)$ holds, we switch $6$ and $5$:
$$\sigm=(11,4,10,2,{\bf 5})(16,9,3,{\bf 6},12)(20,1,7,14,18,8,13,19,15,17).$$
The entries to their left are $2$ and $3$, so they need to be switched, and then the entries preceding these are $10$ and $9$, so they need to be switched as well:
$$\sigm=(11,4,{\bf 9},{\bf 3},5)(16,{\bf 10},{\bf 2},6,12)(20,1,7,14,18,8,13,19,15,17).$$

Since $P(5,4)$ is false, we now look at $\Gamma_2$, so $z=b_2=12$. Only $P(12,13)$ holds, so $\eps=1$ and we switch $12$ and $13$:
$$\sigm=(11,4,9,3,5)(16,10,2,6,{\bf 13})(20,1,7,14,18,8,{\bf 12},19,15,17).$$
Now $z=13$ and $P(13,14)$ holds, so we switch $13$ and $14$, the preceding entries $6$ and $7$, and also $2$ and $1$:
$$\sigm=(11,4,9,3,5)(16,10,{\bf 1},{\bf 7},{\bf 14})(20,{\bf 2},{\bf 6},{\bf 13},18,8,12,19,15,17).$$
Finally, $z=14$ and $P(14,15)$ holds, so we switch $14$ and $15$, and we stop here because $P(15,16)$ is false:
$$\bij(\pi)=\sigm=(11,4,9,3,5)(16,10,1,7,{\bf 15})(20,2,6,13,18,8,12,19,{\bf 14},17)\in\S_{20}.$$
In one-line notation,
\renewcommand{\tabcolsep}{0pt}
$$\begin{array}{cc*{2}{c@{\,\cdot\,}}*{4}{c@{\ }}*{3}{c@{\,\cdot\,}}*{2}{c@{\ }}*{4}{c@{\,\cdot\,}}c@{\ }c@{\,\cdot\,}c@{\ }c@{\,\cdot\,}c@{\ }c}
\pi&=&7&6&5&10&12&14&16&13&3&1&4&20&19&18&16&9&21&8&15&2&11\\
\bij(\pi)&=&7&6&5&9&11&13&15&12&3&1&4&19&18&17&16&10&20&8&14&2&
\end{array}$$
where the descents have been marked with dots.
\end{example}

\begin{example}
Let $$\pi=(2,9,17,6,11,19,7,13,12,15,8,14,1,4,5,10,18,3,16,20)\in\C_{20}.$$
Inserting parentheses before the left-to-right maxima, we have
$$\sigm=(2)(9)(17,6,11)(19,7,13,12,15,8,14,1,4,5,10,18,3,16).$$
Now $z=b_1=2$, and only $P(2,1)$ holds, so we switch $2$ and $1$:
$$\sigm=({\bf 1})(9)(17,6,11)(19,7,13,12,15,8,14,{\bf 2},4,5,10,18,3,16).$$
In $\Gamma_2$ we have $z=b_2=9$ and $P(9,8)$ holds, so we switch $9$ and $8$.
Now $P(8,7)$ holds, so we switch $8$ and $7$. Similarly, we switch $7$ and $6$, then $6$ and $5$, and then $5$ and $4$, obtaining
$$\sigm=(1)({\bf 4})(17,{\bf 7},11)(19,{\bf 8},13,12,15,{\bf 9},14,2,{\bf 5},{\bf 6},10,18,3,16).$$
Finally, in $\Gamma_3$ we have $z=b_3=11$ and $P(11,10)$ holds, so we switch $11$ and $10$, and also the preceding entries $7$ and $6$:
$$\bij(\pi)=\sigm=(1)(4)(17,{\bf 6},{\bf 10})(19,8,13,12,15,9,14,2,5,{\bf 7},{\bf 11},18,3,16)\in\S_{19}.$$
In one-line notation,
\renewcommand{\tabcolsep}{0pt}
$$\begin{array}{cc*{2}{c@{\ }}c@{\,\cdot\,}*{7}{c@{\ }}*{3}{c@{\,\cdot\,}}*{2}{c@{\ }}*{2}{c@{\,\cdot\,}}c@{\ }c@{\,}c}
\pi&=&4&9&16&5&10&11&13&14&17&18&19&15&12&1&8&20&6&3&7&\cdot2\\
\bij(\pi)&=&1&5&16&4&7&10&11&13&14&17&18&15&12&2&9&19&6&3&8
\end{array}$$
where the descents have been marked with dots.
\end{example}

\subsection{The map $\inv$}

Given $\sigma\in\S_{n}$, write it in cycle form with the largest element of each cycle first, ordering the cycles by increasing first element,
say $$\sigma=(c_1,\dots,d_1)(c_2,\dots,d_2)\cdots(c_r,\dots,d_r).$$ Removing the internal parentheses and appending $n+1$, we obtain an $(n+1)$-cycle
$$\pim=(c_1,\dots,d_1;c_2,\dots,d_2;\dots;c_r\dots,d_r;n+1).$$ For convenience we write semicolons in order to keep track of the places from where parentheses were removed.
We call  the $r+1$ subsequences separated by these semicolons {\em blocks} of $\pim$.
Similarly to the description of $\bij$, we will obtain $\inv(\sigma)$ by making some switches to the cycle form of $\pim$.
At each stage of the algorithm, we denote by $\Delta_i$ the $i$-th block of $\pim$.
For $1\le x,y\le n$, let $Q(x,y)$ be the condition
$$\pim(x)>\pim(y) \ \textrm{ and }\ \sigma(x)<\sigma(y).$$

Repeat the following steps for  $i=r{-}1,r{-}2,\dots,1$:
\bit
\item Let $z$ be the rightmost entry of $\Delta_i$. If $Q(z,z+1)$ or $Q(z,z-1)$ holds, let $\eps\in\{-1,1\}$ be such that $Q(z,z+\eps)$ holds and $\pim(z+\eps)$ is smallest.
\item Repeat for as long as $Q(z,z+\eps)$ holds:
\ben\renewcommand{\labelenumi}{\Roman{enumi}'.}
\item Switch $z$ and $z+\eps$ in the cycle form of $\pim$.
\item If the last switch did not involve the leftmost entry of $\Delta_i$, let $x$ and $y$ be the elements preceding the switched entries. If $|x-y|=1$, switch $x$ and $y$ in the cycle form of $\sigma$, and repeat step~II'.
\item Let $z:=z+\eps$ (the new rightmost entry of $\Delta_i$).
\een
\eit

Define $\inv(\sigma)=\pim$.

\ms

\begin{example}
Let $$\sigma=(11,4,9,3,5)(16,10,1,7,15)(20,2,6,13,18,8,12,19,14,17)\in\S_{20}.$$
Removing the parentheses and appending $n+1$ we get $$\pim=(11,4,9,3,5;16,10,1,7,15;20,2,6,13,18,8,12,19,14,17;21).$$

We start looking at $\Delta_2$, so $z=d_2=15$. Only $Q(15,14)$ holds, so we switch $15$ and $14$:
$$\pim=(11,4,9,3,5;16,10,1,7,{\bf 14};20,2,6,13,18,8,12,19,{\bf 15},17;21).$$
Now $z=14$ and $Q(14,13)$ holds, so we switch $14$ and $13$. The entries to their left are $7$ and $6$, and the entries preceding these are $1$ and $2$, so we make the corresponding switches:
$$\pim=(11,4,9,3,5;16,10,{\bf 2},{\bf 6},{\bf 13};20,{\bf 1},{\bf 7},{\bf 14},18,8,12,19,15,17;21).$$
Now $z=13$ and $Q(13,12)$ holds, so we switch $13$ and $12$:
$$\pim=(11,4,9,3,5;16,10,2,6,{\bf 12};20,1,7,14,18,8,{\bf 13},19,15,17;21).$$

Looking at $\Delta_1$, we have $z=d_1=5$. Only $Q(5,6)$ holds, so we switch $5$ and $6$, and also the preceding entries $3$ and $2$, and $9$ and $10$:
$$\pim=(11,4,{\bf 10},{\bf 2},{\bf 6};16,{\bf 9},{\bf 3},{\bf 5},12;20,1,7,14,18,8,13,19,15,17;21).$$
Now $z=6$ and $Q(6,7)$ holds, so we switch $6$ and $7$, and also the preceding entries $2$ and $1$:
$$\pim=(11,4,10,{\bf 1},{\bf 7};16,9,3,5,12;20,{\bf 2},{\bf 6},14,18,8,13,19,15,17;21).$$
Since $Q(7,8)$ is false, the algorithm ends here, so
$$\inv(\sigma)=(11,4,10,1,7,16,9,3,5,12,20,2,6,14,18,8,13,19,15,17,21)\in\C_{21}.$$
\end{example}

\begin{example}
Let $$\sigma=(1)(4)(17,6,10)(19,8,13,12,15,9,14,2,5,7,11,18,3,16)\in\S_{19}.$$
After removing the parentheses, $$\pim=(1;4;17,6,10;19,8,13,12,15,9,14,2,5,7,11,18,3,16;20).$$

In $\Delta_3$, $z=d_3=10$ and $Q(10,11)$ holds, so we switch $10$ and $11$, and also $6$ and $7$:
$$\pim=(1;4;17,{\bf 7},{\bf 11};19,8,13,12,15,9,14,2,5,{\bf 6},{\bf 10},18,3,16;20).$$
Since $Q(11,12)$ is false, we look at $\Delta_2$, so $z=d_2=4$. We see that both $Q(4,3)$ and $Q(4,5)$ hold, but $\pim(3)=16>6=\pim(5)$, so we switch $4$ and $5$.
Now $z=5$ and $Q(5,6)$ holds, so we switch $5$ and $6$. Similarly, we switch $6$ and $7$, next $7$ and $8$, and then $8$ and $9$:
$$\pim=(1;{\bf 9};17,{\bf 6},11;19,{\bf 7},13,12,15,{\bf 8},14,2,{\bf 4},{\bf 5},10,18,3,16;20).$$
Finally, in the first block we switch $1$ and $2$, ending with
$$\inv(\sigma)=\pim=(2,9,17,6,11,19,7,13,12,15,8,14,1,4,5,10,18,3,16,20)\in\C_{20}.$$
\end{example}

\section{Properties of $\bij$ and $\inv$}\label{sec:proof}
In this section we show that $\bij$ preserves descent sets, that $\bij$ and $\inv$ are inverses of each other, and a few other properties of $\bij$.
 The following five lemmas give more insight into the computation of $\bij(\pi)$. They are valid for each $1\le i\le r-1$.
The $i$-th iteration of the main loop of the algorithm will be sometimes referred to as {\em adjusting $\Gamma_i$}.

\begin{lemma}\label{lem:decrpi}
Suppose that in the process of adjusting $\Gamma_i$, the elements that successively occupy the last position of $\Gamma_i$ are $b,b+\eps,b+2\eps,\dots,b+k\eps$ in this order. Then
$$\pi(b)>\pi(b+\eps)>\pi(b+2\eps)>\dots>\pi(b+k\eps).$$
\end{lemma}

\begin{proof}
For each $1\le j\le k$, the switch between $b+(j{-}1)\eps$ and $b+j\eps$ only takes place if $P(b+(j{-}1)\eps,b+j\eps)$ holds, which implies that
$\pi(b+(j{-}1)\eps)>\pi(b+j\eps)$.
\end{proof}

\begin{lemma}\label{lem:relative} \ben \item\label{partt1} No switch ever takes place between two entries of the same $\Gamma_i$.
\item\label{partt2} The relative order of the entries within $\Gamma_i$ always stays the same. In particular, the first entry of $\Gamma_i$ is always the largest.
\een
\end{lemma}

\begin{proof}
To prove part~(\ref{partt1}), assume for contradiction that a switch takes place between two entries of $\Gamma_i$. Consider the first such switch, which must necessarily be
between the last element $z$ and another element $z+\eps$ in $\Gamma_i$, with $\eps\in\{-1,1\}$.
For $P(z,z+\eps)$ to hold, we would need $\sigm(z)<\sigm(z+\eps)$. But this cannot happen because $\sigm(z)$ is the first entry of $\Gamma_i$, and hence the largest since by assumption this is the first switch between two entries of $\Gamma_i$.

Part~(\ref{partt2}) follows from part~(\ref{partt1}) observing that since the switches always involve consecutive values,
the relative order of the entries in $\Gamma_i$ never changes in a switch between an entry of $\Gamma_i$ and an entry of another cycle.
\end{proof}

\begin{lemma}\label{lem:notmoved} While adjusting $\Gamma_i$, \ben
\item\label{part1} neither the first nor the last entry of $\Gamma_j$ is moved for any $j>i$; in particular, before iteration $j$, $\Gamma_j=(a_j,\dots,b_j)$;
\item\label{part2} no entry $t$ with $t\ge a_{i+1}$ is moved;
\item\label{part3} no entry preceding an entry $t\ge a_{i+1}$ is moved.
\een
\end{lemma}

\begin{proof}
We use induction on $i$. If $i>1$, our induction hypothesis is that all three parts of the lemma hold for smaller values of $i$.
We assume we have adjusted $\Gamma_1,\Gamma_2,\dots,\Gamma_{i-1}$, and neither $a_j$ nor $b_j$ for $j\ge i$ have moved.
In particular, $\sigm(b_i)=a_i$ at the start of the $i$-th iteration. If $i=1$, the argument below proves the base case.

Suppose that in the process of adjusting $\Gamma_i$ we move some $b_j$ with $j>i$. Consider the first time this happens,
and let $z$ be the rightmost entry of $\Gamma_i$ right before the switch. Since $b_i$ was the rightmost entry of $\Gamma_i$ before iteration $i$,
we have by Lemma~\ref{lem:decrpi} that $\pi(z)\le\pi(b_i)=a_{i+1}$.
For the switch between $z$ and $b_j$ to happen, we must have $b_j=z\pm1$ and $P(z,b_j)$ must hold,
which implies that $\pi(z)>\pi(b_j)$.  But we know that $\pi(z)\le a_{i+1}$, $\pi(b_j)=a_{j+1}$, and $a_{i+1}<a_{j+1}$
because the left-to-right maxima of a sequence are increasing, so this is a contradiction.

Suppose now that in the process of adjusting $\Gamma_i$ we move some $a_j$ with $j>i$.
Consider the first time this happens. Since switches only take
place between consecutive values and the sequence  $a_1,a_2,\dots$ is increasing, we must have $j=i+1$, and there must be some element in $x$ in $\Gamma_i$ with $|a_{i+1}-x|=1$.
Now, the facts that $a_{i+1}$ is larger than all the elements of $\Gamma_i$ and that $\Gamma_i$ starts with its largest element,
which we know from Lemma~\ref{lem:relative}(\ref{partt2}), imply that $x$ is the first entry of $\Gamma_i$ and $a_{i+1}=x+1$. However, we claim that in this case no switch
takes place. Indeed, for any switch to take place, $P(z,z+\eps)$ must hold for some $\eps\in\{-1,1\}$, where $z$ is the last entry of $\Gamma_i$. This means that
$\pi(z+\eps)<\pi(z)\le\pi(b_i)=a_{i+1}=x+1$ and $\sigm(z+\eps)>\sigm(z)=x$, so $\pi(z+\eps)<\sigm(z+\eps)$. For this to hold,
$z+\eps$ or the entry following it, namely $\sigm(z+\eps)$, have been moved in a previous step of the algorithm. But the fact that $\sigm(z+\eps)\ge a_{i+1}$ makes this impossible, by the induction hypothesis on parts~(\ref{part2}) and~(\ref{part3}).

Now we prove part~(\ref{part2}). At the beginning of the computation of $\bij(\pi)$, when $\sigm$ is given by Equation~(\ref{eq:sigm}), $a_{i+1}$ is larger than all the elements in $\Gamma_1,\dots,\Gamma_i$. Since all the switches involve
consecutive values, no $t\ge a_{i+1}$ can be involved in a switch with elements of $\Gamma_1,\dots,\Gamma_i$ without $a_{i+1}$ being involved in a switch first. But we just proved that this cannot happen.

To prove part~(\ref{part3}), assume without loss of generality that the entry $s$ preceding $t$ is moved for the first time while adjusting $\Gamma_i$.
Since $t$ has not been moved, the switch must happen in step I of the $i$-th iteration, and it must involve $s$ and the last entry of $\Gamma_i$ at the time, say $z$. For this switch to take place,
we need $\pi(z)>\pi(s)$. But this cannot happen because $\pi(z)\le\pi(b_i)=a_{i+1}$, and since neither $s$ nor $t$ have moved so far, $\pi(s)=t\ge a_{i+1}$.
\end{proof}

\begin{lemma}\label{lem:notback} While adjusting $\Gamma_i$, no entries in cycles $\Gamma_j$ with $j<i$ are moved.
\end{lemma}

\begin{proof} Suppose this is false, and consider the first time that the last entry $z$ of $\Gamma_i$ is switched with an entry $z+\eps$ of $\Gamma_j$, where $j<i$.
Then we must have $\sigm(z)<\sigm(z+\eps)$. But $\sigm(z)$ is the first entry of $\Gamma_i$, which is larger than any element in $\Gamma_1,\Gamma_2,\dots,\Gamma_{i-1}$, using Lemma~\ref{lem:notmoved}(\ref{part1})
and the fact that all switches involve consecutive values. In particular, $\sigm(z)$ is larger than $\sigm(z+\eps)$.
\end{proof}

\begin{lemma}\label{lem:bi}
In the process of adjusting $\Gamma_i$, the elements that successively occupy the last position of $\Gamma_i$ are $b_i,b_i+\eps,b_i+2\eps,\dots,b_i+k\eps$ for some $k$, in this order. Additionally,
$$\pi(b_i)>\pi(b_i+\eps)>\pi(b_i+2\eps)>\dots>\pi(b_i+k\eps).$$
\end{lemma}

\begin{proof}
By Lemma~\ref{lem:notmoved}(\ref{part1}), the last position of $\Gamma_i$ at the start of the $i$-th iteration is $b_i$.
The rest follows easily from the description of $\bij$ and Lemma~\ref{lem:decrpi}.
\end{proof}

Now we prove the main property of $\bij$, namely that it preserves the descent set if we forget $\pi(n+1)$.

\begin{prop}\label{prop:descents}
Let $\pi\in\C_{n+1}$ and $\sigma=\bij(\pi)$. Then $$D(\pi)\cap[n-1]=D(\sigma).$$
\end{prop}

\begin{proof}
First observe that if $\sigm$ is the permutation in Equation~(\ref{eq:sigm}), before any cycles are adjusted, then $\sigm(x)=\pi(x)$ for all $x\notin\{b_1,\dots,b_r\}$, and $\sigm(b_i)<\pi(b_i)$ for $1\le i\le r$.

Let $W=(D(\pi)\cap[n-1])\bigtriangleup D(\sigm)$
be the set of indices where the descents of $\pi$ and $\sigm$ disagree ($\bigtriangleup$ denotes the symmetric difference). Before adjusting any cycles, the only indices that may be in $W$
are $b_i-1$ and $b_i$ for $1\le i\le r-1$, by the previous observation.
We claim that adjusting cycle $\Gamma_i$ removes $b_i-1$ and $b_i$ from $W$ (if they were in it) without adding any other elements to $W$.
Indeed, by Lemma~\ref{lem:bi}, the first step in iteration $i$ checks $P(b_i,b_i-1)$ and $P(b_i,b_i+1)$, which determine whether $b_i-1$ and $b_i$ are in $W$, respectively. If either of them is, the
switch between $b_i$ and $b_i+\eps$ (with $\eps$ chosen so that $\sigm(b_i+\eps)$ is largest) performed by $\bij$ in step I of the $i$-th iteration guarantees that $b_i-1,b_i\notin W$ after the switch.
However, two elements could now have been added to $W$:\bit
\item If $b_i$ was not the first entry of $\Gamma_i$ (note that by Lemma~\ref{lem:notmoved}(\ref{part1}) we know that $b_i+\eps$ was not the first entry of its cycle)
and the entries preceding $b_i$ and $b_i+\eps$ were consecutive, say $s$ and $s+1$,
then the switch between $b_i$ and $b_i+\eps$ adds $s$ to $W$. Step II of the $i$-th iteration switches $s$ and $s+1$ so that $s$ is no longer in $W$, and next
performs any necessary switches to prevent any other indices from being added to $W$.
\item It is possible that since $\sigm(b_i+\eps)$ has changed in step I, the relative order of $\sigm(b_i+\eps)$ and $\sigm(b_i+2\eps)$ is now different from the relative order of $\pi(b_i+\eps)$ and $\pi(b_i+2\eps)$.
The condition $P(b_i+\eps,b_i+2\eps)$ determines whether this is the case, and if so, the second repetition of step I
switches $b_i+\eps$ and $b_i+2\eps$ in the cycle form of $\sigm$ to fix the problem. Again, step II prevents other elements from being added to $W$.
\eit
These steps are repeated until for some $k$, $P(b_i+k\eps,b_i+(k{+}1)\eps)$ is false, which means that either $b_i+(k{+}1)\eps\in\{0,n+1\}$
or the relative order of $\sigm(b_i+k\eps)$ and $\sigm(b_i+(k{+}1)\eps)$ agrees with the relative order of $\pi(b_i+k\eps)$ and $\pi(b_i+(k{+}1)\eps)$.
Iteration $i$ ends here; at this time, the descent set of the sequence
$\sigm(b_i-\eps)\sigm(b_i)\sigm(b_i+\eps)\dots\sigm(b_i+k\eps)\sigm(b_i+(k{+}1)\eps)$ agrees with the descent set of
$\pi(b_i-\eps)\pi(b_i)\pi(b_i+\eps)\dots\pi(b_i+k\eps)\pi(b_i+(k{+}1)\eps)$, and the only elements that may remain in $W$ are $b_j-1$ and $b_j$ for $i< j\le r-1$.
After iteration $r-1$, we have $W=\emptyset$, so the result is proved.
\end{proof}

The next result describes two more properties of our bijection. Recall that $E(\tau)$ denotes the set of weak excedances of a permutation $\tau$.
We remark that $\pi\in\C_{n+1}$ has no fixed points when $n\ge1$, so its weak excedances are just its excedances, namely indices $i$ with $\pi(i)>i$.

\begin{prop}\label{prop:properties}
Let $\pi\in\C_{n+1}$ and $\sigma=\bij(\pi)$, where $n\ge1$. Then \ben
\item\label{n} $\pi^{-1}(n+1)=\sigma^{-1}(n)$,
\item\label{wexc} $E(\pi)=E(\sigma)$.
\een
\end{prop}

\begin{proof}
To prove part~(\ref{n}), note that if $\sigm$ is the permutation in Equation~(\ref{eq:sigm}), then
$\pi(b_r)=n+1$ and $\sigm(b_r)=a_r=n$. By Lemma~\ref{lem:notmoved}(\ref{part1}), $a_r$ and $b_r$ are never moved when adjusting the cycles $\Gamma_1,\dots,\Gamma_{r-1}$, so $\sigma(b_r)=n$.
It follows that $\pi^{-1}(n+1)=b_r=\sigma^{-1}(n)$.

To prove part~(\ref{wexc}), we first observe that before any cycles are adjusted in $\sigm$, we have that $\sigm(x)=\pi(x)$ for all $x\notin\{b_1,\dots,b_r\}$ with $1\le x\le n$. We have $\pi(n+1)<n+1$,
and for $1\le i\le r$, $\pi(b_i)=a_{i+1}$ and $\sigm(b_i)=a_i$, so $b_i$ is a weak excedance in both $\pi$ and $\sigm$. Thus $E(\pi)=E(\sigm)$.

Next we show that the set of weak excedances of $\sigm$ is preserved in each switch performed by $\bij$.
Let $1\le x\le n$, and let $\sigm$ and $\sigm'$ the permutations before and after such a switch, respectively.
We claim that $x$ is a weak excedance of $\sigm$ if and only if it is a weak excedance of $\sigm'$. Here are all the possible cases:
\bit
\item If the switch involves neither $x$ nor $\sigm(x)$, the claim is trivial because $\sigm(x)=\sigm'(x)$.
\item If the switch involves $\sigm(x)$ but not $x$ (in particular $\sigm(x)\neq x$), then $\sigm(x)$ must be switched with $\sigm(x)\pm1$, since all switches involve consecutive values,
so $\sigm'(x)=\sigm(x)\pm1$. But since $x$ is not involved in the switch, $x\neq\sigm'(x)$, so $x\in E(\sigm)$ if and only if $x\in E(\sigm')$.
\item If the switch involves $x$, let us consider two cases according to whether the switch is done in step I or II of the algorithm.
If it takes place in step I, it means that either $x$ or the entry that it is switched with, which must be $x\pm1$,
is the rightmost entry $z$ of a cycle. We show that in this case, the two switched entries $z$ and $z+\eps$ are weak excedances both before and after the switch, so $E(\sigm)=E(\sigm')$.
Indeed, the rightmost entry of a cycle is always a weak excedance by Lemma~\ref{lem:relative}(\ref{partt2}), so $z\in E(\sigm)$ and $z+\eps\in E(\sigm')$.
On the other hand, for $P(z,z+\eps)$ to hold, $z\le \sigm(z)<\sigm(z+\eps)$, so $\sigm(z+\eps)\ge z+\eps$ and $z+\eps\in E(\sigm)$
(in fact $\sigm(z+\eps)> z+\eps$ because by Lemmas~\ref{lem:notmoved}(\ref{part1}) and~\ref{lem:notback}, $z+\eps$ cannot be a fixed point).
Finally, $\sigm'(z)=\sigm(z+\eps)>z$, so $z\in E(\sigm')$.\\
 If the switch takes place in step II, then $\sigm'(x)=\sigm(x)\pm1$ and $x$ is not a fixed point of neither $\sigm$ nor $\sigm'$, so again $x\in E(\sigm)$ if and only if $x\in E(\sigm')$.
In fact, if $\sigm''$ was the permutation right before the previous switch, we have $\sigm''(x)=\sigm'(x)$.
\eit
\end{proof}

Our last goal in this section is to show that $\bij$ and $\inv$ are inverses of each other. For this purpose, we introduce some notation and a few lemmas leading to the proof of this fact,
which is the statement of Proposition~\ref{prop:inverses} below.

Let $\pi\in\C_{n+1}$ and $\sigma=\bij(\pi)$. We will prove that $\inv(\sigma)=\pi$ by showing that the switches performed by the $i$-th iteration of $\bij$ on $\pi$
(that is, while $\Gamma_i$ is adjusted) are the same in reverse order as the switches performed by the iteration of $\inv$ on $\sigma$ corresponding to the same $i$.
We will call this the $i$-th iteration of $\inv$, despite the fact that this convention makes the iteration number in $\inv$ decrease from $r-1$ to $1$.
For the rest of this section, the value of $i$ is fixed, with $1\le i\le r-1$, and we focus on iteration $i$ of $\bij$ (i.e., of the algorithm that computes $\bij(\pi)$)
and $\inv$ (i.e., of the algorithm that computes $\inv(\sigma)$).

By Lemma~\ref{lem:bi}, the elements that successively occupy the last position of $\Gamma_i$ during iteration $i$ of $\bij$ are $b_i,b_i+\eps,b_i+2\eps,\dots,b_i+k\eps$ for some $k\ge0$.
For $1\le j\le k$, we call {\it subiteration $j$} (of iteration $i$) the set of switches that begin with
the switch between $b_i+(j{-}1)\eps$ and $b_i+j\eps$ (step I) followed by the subsequent switches of the preceding entries (step II).
In general, subiteration $j$ switches a set of adjacent entries of $\Gamma_i$, the rightmost one of which is $b_i+(j{-}1)\eps$,
with a set of adjacent entries, the rightmost one of which is $b_i+j\eps$.
This latter set of adjacent entries must be outside and to the right of $\Gamma_i$ by Lemmas~\ref{lem:relative}(\ref{partt1}) and~\ref{lem:notback}, and must
belong to the same $\Gamma_\ell$ with $\ell>i$ by Lemma~\ref{lem:notmoved}(\ref{part1}).
The set of (adjacent) positions of these entries will be called the $j$-th {\em cluster} and denoted~$R_j$.

If $k=0$, no changes are made in iteration $i$ of $\bij$. Let us now show that in this simple case, no switches take place in iteration $i$ of $\inv$ either.
Indeed, for $Q(b_i,b_i+\eps)$ to hold, we would need $\pim(b_i)>\pim(b_i+\eps)$ and $\sigma(b_i)<\sigma(b_i+\eps)$.
By Lemma~\ref{lem:notmoved}(\ref{part1}), the first inequality can be written as $a_{i+1}>\pim(b_i+\eps)$. By Proposition~\ref{prop:descents}, the second inequality is equivalent to
$a_{i+1}=\pi(b_i)<\pi(b_i+\eps)$. But this implies, using Lemma~\ref{lem:notmoved}, that neither $\pi(b_i+\eps)$ nor $b_i+\eps$ have moved before iteration $i$ of $\bij$, so it contradicts that $a_{i+1}>\pim(b_i+\eps)$.
Similarly, $Q(b_i,b_i-\eps)$ does not hold either, so no switches take place in iteration $i$ of $\inv$.
In the rest of this section we will assume that $k>0$.

Denote by $\sigm_0$ the permutation $\sigm$ right before the $i$-th iteration of $\bij$ starts.
For each $1\le j\le k$, denote by $\sigm_j$ the permutation $\sigm$ right after subiteration $j$. For $j<k$,
this is right before the rightmost entry $b_i+j\eps$ of $\Gamma_i$ is switched with $b_i+(j{+}1)\eps$.
Note that $\sigm_k$ is the permutation $\sigm$ at the end of the $i$-th iteration.
For $0\le j\le k$, let $a_i^j$ denote the leftmost entry of $\Gamma_i$ in $\sigm_j$, and let
$\pim_j\in\C_{n+1}$ be the permutation whose cycle notation is obtained by removing
all but the first and last parentheses in the cycle form of $\sigm_j$ and appending $n+1$.
Note that by Lemma~\ref{lem:notmoved}(\ref{part1}), $a_i^0=a_i$. For $0\le j\le k$, we have that $\sigm_j(b_i+j\eps)=a_i^j$, because
$a_i^j$ and $b_i+j\eps$ are the leftmost and rightmost entry of $\Gamma_i$ in $\sigm_j$, respectively.

We will show that if $\pim=\pim_k$ right before the $i$-th iteration of $\inv$, then $\pi=\pim_0$ right after the $i$-th iteration,
so the $i$-th iteration of $\inv$ undoes the switches performed by the $i$-th iteration of $\bij$.

Using the same notation as in the description of $\bij$, we have $$\pi=(a_1,\dots,b_1,a_2,\dots,b_2,\dots,a_r,\dots,b_r,n+1).$$
The $i$-th cycle of $\sigm_0$ is $\Gamma_i=(a_i,\dots,b_i)$, because by Lemma~\ref{lem:notmoved}(\ref{part1}), $a_i$ and $b_i$ have not been moved before iteration $i$ of $\bij$.
For $1\le j\le k$, let $s_j=\sigm_{j-1}(b_i+j\eps)$. If $b_i+(k{+}1)\eps\notin\{0,n+1\}$, let $s_{k+1}=\sigm_k(b_i+(k{+}1)\eps)$, and if $b_i-\eps\notin\{0,n+1\}$, let $s_0=\sigm_0(b_i-\eps)$.

In Lemmas~\ref{obsa} through~\ref{obsd} below, unless otherwise stated, $j$ is an arbitrary integer with $1\le j\le k$.
Using the terminology from the above paragraphs,
subiteration $j$
starts with $\sigm_{j-1}$ and
switches a segment of adjacent entries of $\Gamma_i$, including its rightmost entry $b_i+(j{-}1)\eps$,
with the entries in the cluster $R_j$, the rightmost one being $b_i+j\eps$. Entries that are switched with each other differ by $\pm1$. Here is a schematic representation of subiteration $j$, with the underbrackets
indicating the entries that are switched:
$$\begin{array}{rccccc} \sigm_{j-1}=& \dots &\overbrace{(a_i^{j-1},\dots,\underbrace{x,\dots, b_i+(j{-}1)\eps})}^{\Gamma_i}& \dots &(\dots,\underbrace{x{\pm}1,\dots, b_i+j\eps}_{R_j}, s_j, \dots)& \dots \\
\sigm_{j}=& \dots &(a_i^{j},\dots,\underbrace{x{\pm}1,\dots, b_i+j\eps})& \dots  &(\dots,\underbrace{x,\dots, b_i+(j{-}1)\eps}_{R_j}, s_j, \dots)& \dots
\end{array}$$
Note that $a_i^{j-1}=a_i^j$ unless all the entries of $\Gamma_i$ are involved in the switch.

\begin{lemma}\label{obsa}
We have
$$\pi(b_i+(j{-}1)\eps)>\pi(b_i+j\eps)\ \textrm{ and }\ a_i^{j-1}<s_j.$$
\end{lemma}

\begin{proof}
For the switch between $b_i+(j{-}1)\eps$ and $b_i+j\eps$ to take place in subiteration $j$, $P(b_i+(j{-}1)\eps,b_i+j\eps)$ must hold, that is,
$\pi(b_i+(j{-}1)\eps)>\pi(b_i+j\eps)$ and $\sigm_{j-1}(b_i+(j{-}1)\eps)<\sigm_{j-1}(b_i+j\eps)$. The second inequality can be written as $a_i^{j-1}<s_j$.
\end{proof}

\begin{lemma}\label{lem:smaller}
Every entry in $\sigm_{j-1}$ and $\sigm_j$ that is inside (but not in the leftmost position of) the cluster $R_j$ is smaller than $a_i^{j-1}$.
The leftmost entry of $R_j$ is at most $a_i^{j-1}+1$ in $\sigm_{j-1}$ and at most $a_i^{j-1}$ in $\sigm_j$.
\end{lemma}

\begin{proof}
By Lemma~\ref{lem:relative}(\ref{partt2}), $a_i^{j-1}$ is the largest entry of $\Gamma_i$ in $\sigm_{j-1}$. In $\sigm_j$, the cluster $R_j$ contains entries that were
in $\Gamma_i$ in $\sigm_{j-1}$, so the statements about $\sigm_j$ follow.

To prove the statements about $\sigm_{j-1}$, note that at the start of subiteration $j$,
the value of each entry of $R_j$ must equal the value of the entry of $\Gamma_i$ that it is switched with, plus or minus one.
\end{proof}

\begin{lemma}\label{lem:clusters}
In the $i$-th iteration of $\bij$, the following statements hold:
\ben
\item\label{pd} The leftmost and rightmost positions of each $\Gamma_\ell$ with $\ell>i$ do not belong to any cluster. In particular, $b_i+j\eps$ is followed by $s_j$ in $\sigm_{j-1}$.
\item\label{pa} The clusters $R_1,\dots,R_k$ do not overlap.
\item\label{pb} Any position in the cycles $\Gamma_\ell$ with $\ell>i$ is involved in at most one switch.
\item\label{pc} Each $b_i+j\eps$ with $1\le j\le k-1$ is moved exactly twice: first, in subiteration $j$, from the rightmost position of $R_j$ to the rightmost position of $\Gamma_i$, and then,
in subiteration $j+1$, from there to the rightmost position of $R_{j+1}$.\\ 
The entries $b_i$ and $b_i+k\eps$ are moved only once: $b_i$ is moved in subiteration $1$ from the rightmost position of $\Gamma_i$ to the rightmost position of $R_1$;
$b_i+k\eps$ is moved in subiteration $k$ from the rightmost position of $R_k$ to the rightmost position of~$\Gamma_i$.
\item\label{pe} For each fixed position in $\Gamma_i$, the values of the entries that occupy that position may go up or down during iteration $i$ of $\bij$, but they cannot do both.
\een
\end{lemma}

\begin{proof}
By Lemma~\ref{lem:notmoved}(\ref{part1}), none of the $a_\ell,b_\ell$ with $\ell>i$ are moved while adjusting $\Gamma_i$, so they do not belong to any cluster, proving part~(\ref{pd}).

For part~(\ref{pa}), suppose that some clusters overlap and, among those, let $R_p$ and $R_q$ be the ones whose rightmost positions are furthest to the right. Then $R_p$ and $R_q$ overlap,
and they cannot have the same rightmost position, because by Lemma~\ref{lem:bi}, an entry that is taken out of the rightmost position of $\Gamma_i$ during the $i$-th iteration is not put back in.
Assume that the rightmost position of $R_p$ is to the right of $R_q$.

If $p<q$, then the entry $s_q$ following $b_i+q\eps$ in $\sigm_{q-1}$ has moved there in subiteration $p$ (since it is inside $R_p$), and it is not switched again in iteration $i$, by the choice of $p$ and $q$.
The entry $s_q$ in $\sigm_p$ is not the leftmost entry of $R_p$, because $R_p$ and $R_q$ overlap, so
by Lemma~\ref{lem:smaller}, $s_q<a_i^{p-1}$. Since $s_q$ does not move again in iteration $i$ and the first entry of $\Gamma_i$ can only change by consecutive values, we also have $s_q<a_i^{q-1}$, contradicting Lemma~\ref{obsa}.

If $p>q$, then $s_q$ is not moved during the $i$-th iteration until subiteration $p$, by the choice of $p$ and $q$. Since
$s_q>a_i^{q-1}$ by Lemma~\ref{obsa} the first entry of $\Gamma_i$ can only change by consecutive values,
it follows that $s_q>a_i^{p-1}$ as well. But this contradicts Lemma~\ref{lem:smaller}, because $s_q$ belongs to $R_p$ in $\sigm_{p-1}$.

Part~(\ref{pb}) is a trivial consequence of part~(\ref{pa}), since each cluster is involved in only one switch in iteration $i$.
 Part~(\ref{pc}) follows immediately from the previous parts and the definitions of $\bij$ and the clusters.

To prove part~(\ref{pe}), note that if the value of an entry goes up and then back down, then it must repeat a value, since switches move the values by $\pm1$. This would imply that some position to the right of $\Gamma_i$ is involved in more than one switch, contradicting part~(\ref{pb}).
\end{proof}

It follows from Lemma~\ref{lem:clusters}(\ref{pe}) that the values of the leftmost entry of $\Gamma_i$ during iteration $i$ of $\bij$ satisfy one of these conditions:
\bit \item $a_i^0=a_i^1=\dots=a_i^k$ (we say that iteration $i$ has {\em type H}),
\item $a_i^0\le a_i^1\le \dots\le a_i^k$ (we say it has {\em type U}, unless it has type H),
\item $a_i^0\ge a_i^1\ge \dots\ge a_i^k$ (we say it has {\em type D}, unless it has type H).
\eit

\begin{lemma}\label{obsd}
If the entry $s_j$ is moved in the $i$-th iteration of $\bij$, then it is switched with the leftmost entry of $\Gamma_i$.
\end{lemma}

\begin{proof}
Recall that $s_j$ is the first entry outside and to the right of $R_j$ in $\sigm_{j-1}$. By Lemma~\ref{obsa}, $a_i^{j-1}<s_j$. Assume that $s_j$ is moved in subiteration $p$ for some $p$.

By Lemma~\ref{lem:clusters}(\ref{pa}), the blocks $R_j$ and $R_p$ must be adjacent, with $R_j$ to the left of $R_p$. By Lemma~\ref{lem:clusters}(\ref{pb}),
the leftmost position of $R_p$, which is the one containing $s_j$, is only involved in the switch in subiteration $p$.

Consider first the case $j<p$. Since $a_i^{j-1}<s_j$ and the value of the first entry of $\Gamma_i$ can only change by one at a time,
we have that $a_i^{p-1}<s_j$. But by Lemma~\ref{lem:smaller} applied to $R_p$, noting that $s_j$ is the leftmost entry of $R_p$ in $\sigm_{p-1}$, we have
$s_j\le a_i^{p-1}+1$, from where $s_j=a_i^{p-1}+1$. Thus, $s_j$ is switched with the first entry of $\Gamma_i$, namely $a_i^{p-1}=s_j-1$, which becomes $a_i^p=s_j$ after the switch. Note that in this case iteration $i$ has type U.

Suppose now that $j>p$. In this case, $s_j$ is moved from $\Gamma_i$ to $R_p$ in subiteration $p$, and is not moved again in iteration $i$. Then we must have $a_i^{p}<s_j$, given that $a_i^{j-1}<s_j$, by Lemma~\ref{obsa},
and that the value of the first entry of $\Gamma_i$ can only change by one at a time. Again by Lemma~\ref{lem:smaller} applied to $R_p$, $s_j\le a_i^{p-1}$. Combining the inequalities $a_i^{p}<s_j\le a_i^{p-1}$, we see that
$s_j=a_i^{p-1}=a_i^p+1$, which means that $s_j$ came from the leftmost position in $\Gamma_i$ in subiteration $p$. Note that in this case iteration $i$ has type D.
\end{proof}

\begin{lemma}\label{obsf}
Let $0\le j,\ell\le k$. For $\ell\neq j$, we have $\pim_\ell(b_i+j\eps)=\sigm_\ell(b_i+j\eps)$.
Also, $\pim_j(b_i+j\eps)=a_{i+1}$ and $\sigm_j(b_i+j\eps)=a_i^j$.
\end{lemma}

\begin{proof}
By Lemma~\ref{lem:clusters}(\ref{pc}), $b_i+j\eps$ is in the rightmost position of $R_j$ in $\sigm_\ell$ when $\ell<j$, and in the rightmost position of $R_{j+1}$ when $\ell>j$.
By Lemma~\ref{lem:clusters}(\ref{pd}), an entry in a cluster is never the rightmost position of a cycle $\Gamma_\ell$. Thus, the entry following $b_i+j\eps$ in $\sigm_\ell$ when $\ell\neq j$ equals
both $\pim_\ell(b_i+j\eps)$ and $\sigm_\ell(b_i+j\eps)$.

To prove the second sentence, note that in $\sigm_j$, $b_i+j\eps$ is the rightmost entry of $\Gamma_i$, $a_i^j$ is the leftmost entry of $\Gamma_i$, and $a_{i+1}$ is the leftmost entry of $\Gamma_{i+1}$, by
Lemma~\ref{lem:notmoved}(\ref{part1}).
\end{proof}

In the next two lemmas we consider the case where iteration $i$ has type H. This case is simpler because $a_i$ is not moved, and therefore, by Lemma~\ref{obsd}, none of the $s_j$ with $1\le j\le k$ is moved during iteration $i$.

\begin{lemma}\label{lem:typeH1}
If the $i$-th iteration of $\bij$ has type H, then
\ben \item\label{orderai}
$a_{i+1}>s_1>s_2>\dots>s_k>a_i$, %\eeq
\item \label{s0} $s_0$ does not satisfy $a_{i+1}>s_0>s_1$,
\item \label{sk1} $s_{k+1}$ does not satisfy $s_k>s_{k+1}>a_i$.
\een
\end{lemma}

\begin{proof}
Together with Lemma~\ref{lem:clusters}(\ref{pc}), the above observation that the $s_j$ are not moved implies that $\sigm_\ell(b_i+j\eps)=s_j$ for $0\le\ell\le j-1$ and $\sigm_\ell(b_i+j\eps)=s_{j+1}$ for $j+1\le\ell\le k$.
Note also that $\sigm_j(b_i+j\eps)=a_i^j=a_i$ for $0\le j\le k$.
By Lemma~\ref{obsa}, $a_i=a_i^{j-1}<s_j$, and then by Lemma~\ref{lem:notmoved}(\ref{part2})(\ref{part3}), neither $b_i+j\eps$ nor $s_j$ for $1\le j\le k$
have been moved in the first $i-1$ iterations of $\bij$, so $\pi(b_i+j\eps)=s_j$. Part~(\ref{orderai}) now follows from Lemma~\ref{lem:bi} and the fact that $\pi(b_i)=a_{i+1}$.

To prove part~(\ref{s0}), assume that $s_0=\sigm_0(b_i-\eps)$ is defined and that $a_{i+1}>s_0>s_1$. In particular, $a_i<s_1<a_{i+1}$, so $s_0$ is not the first entry of a cycle and we have
\beq\label{eq:s0} a_i=\sigm_0(b_i)<\sigm_0(b_i-\eps)=s_0 \mbox{ and } a_{i+1}=\pi(b_i)>\pi(b_i-\eps)=s_0,\eeq
where the last equality follows again from Lemma~\ref{lem:notmoved}(\ref{part2})(\ref{part3}) using that $a_i<s_0$ and so neither $b_i-\eps$ nor $s_0$
have been moved in the first $i-1$ iterations of $\bij$.  Equation~(\ref{eq:s0}) implies that $P(b_i,b_i-\eps)$ would hold in this case, and since $s_0>s_1$,
the algorithm would have switched $b_i$ with $b_i-\eps$ instead of with $b_i+\eps$.

Similarly, to prove part~(\ref{sk1}), assume that $s_{k+1}=\sigm_k(b_i+(k{+}1)\eps)$ is defined and that $s_k>s_{k+1}>a_i$. We claim that in this case,
\beq\label{eq:sk1} a_i=\sigm_k(b_i+k\eps)<\sigm_k(b_i+(k{+}1)\eps)=s_{k+1} \mbox{ and } s_k=\pi(b_i+k\eps)>\pi(b_i+(k{+}1)\eps)=s_{k+1},\eeq
which implies that $P(b_i+k\eps,b_i+(k{+}1)\eps)$ would hold, and the algorithm would switch $b_i+k\eps$ with $b_i+(k{+}1)\eps$,
instead of ending iteration $i$ right after subiteration $k$. The only statement in
Equation~(\ref{eq:sk1}) that does not follow immediately from the definitions and the above observations is the last equality.
We first note that $b_i+(k{+}1)\eps$ is not the rightmost entry of a cycle in $\sigm_k$, otherwise
$s_{k+1}$ (which is $\sigm_k(b_i+(k{+}1)\eps)$ by definition) would be the first entry of the cycle and thus $s_{k+1}\ge a_{i+1}>s_k$.
To prove the last equality in Equation~(\ref{eq:sk1}) it is enough to show that neither $b_i+(k{+}1)\eps$ nor $s_{k+1}$ are moved in the first $i$ iterations of $\bij$.
We first use Lemma~\ref{lem:notmoved}(\ref{part2})(\ref{part3}) and the fact that $a_i<s_{k+1}$ to deduce that
neither $s_{k+1}$ nor the entry preceding it have been moved in the first $i-1$ iterations of $\bij$. Also, $s_{k+1}$ is not moved during iteration
$i$ because it is larger than the first entry $a_i$ of $\Gamma_i$. But then the entry preceding $s_{k+1}$ could only have moved during iteration $i$ if it was the rightmost entry of a cluster,
which is not the case because the entry preceding $s_{k+1}$ in $\sigm_k$ is $b_i+(k{+}1)\eps$.
\end{proof}

\begin{lemma}\label{lem:bijinvH}
Suppose that the $i$-th iteration of $\bij$ has type H.
Then iteration $i$ of $\inv$ undoes precisely the switches performed by iteration $i$ of $\bij$.
\end{lemma}

\begin{proof}
Suppose that $\pim=\pim_k$ right before the $i$-th iteration of $\inv$. The $i$-th block of $\pim_k$ is $c_i,\dots,d_i$,
where $c_i=a_i$ and $d_i=b_i+k\eps$. At this point, $Q(b_i+k\eps,b_i+(k{-}1)\eps)$ is the condition $$\pim_k(b_i+k\eps)>\pim_k(b_i+(k{-}1)\eps)\ \textrm{ and }\ \sigma(b_i+k\eps)<\sigma(b_i+(k{-}1)\eps).$$
The first inequality can be restated as $a_{i+1}>\sigm_k(b_i+(k{-}1)\eps)=s_k$ by Lemma~\ref{obsf}, and it holds by Lemma~\ref{lem:typeH1}(\ref{orderai}).
The second inequality is equivalent to $\pi(b_i+k\eps)<\pi(b_i+(k{-}1)\eps)$ by Proposition~\ref{prop:descents},
and it holds by Lemma~\ref{lem:bi}. Additionally, since $s_k>s_{k+1}>a_i$ does not hold by Lemma~\ref{lem:typeH1}(\ref{sk1}),
either $Q(b_i+k\eps,b_i+(k{+}1)\eps)$ does not hold or, if it does, then $s_k=\pim_k(b_i+(k{-}1)\eps)<\pim_k(b_i+(k{+}1)\eps)=s_{k+1}$. In either case, the $i$-th iteration of $\inv$ starts
 $i$ by switching $b_i+k\eps$ and $b_i+(k{-}1)\eps$ (as opposed to $b_i+(k{+}1)\eps$) in step I'.
Next, the switches in step II' of $\inv$ undo the switches from step II of subiteration $k$ of $\bij$, restoring cluster $R_k$.

Afterwards, for each $j=k{-}1,k{-}2,\dots,1$, the computation of $\inv$ checks condition $Q(b_i+j\eps,b_i+(j{-}1)\eps)$, that is, whether $$\pim_j(b_i+j\eps)>\pim_j(b_i+(j{-}1)\eps)\ \textrm{ and }\ \sigma(b_i+j\eps)<\sigma(b_i+(j{-}1)\eps).$$
Again, the first inequality can be restated as $a_{i+1}>\sigm_j(b_i+(j{-}1)\eps)=s_j$ by Lemma~\ref{obsf}, and it holds by Lemma~\ref{lem:typeH1}(\ref{orderai}). The second inequality is equivalent to $\pi(b_i+j\eps)<\pi(b_i+(j{-}1)\eps)$ by Proposition~\ref{prop:descents}, and it holds by Lemma~\ref{lem:bi}. Thus, $\inv$ performs the switch between $b_i+j\eps$ and $b_i+(j{-}1)\eps$ in step I', followed by the switches in step II' that undo the ones performed in subiteration $j$ of $\bij$, restoring cluster $R_j$.

Finally, $\inv$ checks condition $Q(b_i,b_i-\eps)$, that is, whether \beq\label{eq:pim0}\pim_0(b_i)>\pim_0(b_i-\eps)\ \textrm{ and }\ \sigma(b_i)<\sigma(b_i-\eps),\eeq assuming that $b_i-\eps\notin\{0,n+1\}$.
We have $\pim_0(b_i)=a_{i+1}$ by Lemma~\ref{obsf}, and $\pim_0(b_i-\eps)\ge\sigm_0(b_i-\eps)=s_0$ (with strict inequality if $b_i-\eps$ is at the end of a cycle).
The second inequality in~(\ref{eq:pim0})
is equivalent to $\sigm_1(b_i)<\sigm_1(b_i-\eps)$ because, as shown in the proof of Proposition~\ref{prop:descents},
$\sigm_1$ has a descent between positions $b_i$ and $b_i-\eps$ if and only if $\sigma$ (equivalently, $\pi$) does. Since $\sigm_1(b_i)=s_1$ and $\sigm_1(b_i-\eps)$ equals $\sigm_0(b_i-\eps)=s_0$, possibly plus or minus one (if $s_0$
was involved in the switch in subiteration $1$), this inequality implies that $s_1<s_0$.
So, if both inequalities~(\ref{eq:pim0}) held, then $a_{i+1}>s_0>s_1$, contradicting Lemma~\ref{lem:typeH1}(\ref{s0}). Thus, iteration $i$ of $\inv$ stops here.
\end{proof}

When the $i$-th iteration of $\bij$ has type U or D, the conclusion from Lemma~\ref{lem:bijinvH} still holds, but some of the arguments in the proof, including the statement of Lemma~\ref{lem:typeH1}, have to be slightly modified
to take into account the fact that some of the $s_j$ can be moved during iteration $i$. By Lemma~\ref{obsd}, $s_j$ can only be moved if it is switched with the leftmost entry of $\Gamma_i$. For that to happen, block $R_j$ must have a block $R_p$
immediately to its right.
The leftmost entry of $R_p$ then switches from $s_j$ to $s_j-1$ (resp. $s_j+1$) if $p>j$, or from $s_j+1$ (resp. $s_j-1$) to $s_j$ if $p<j$, assuming iteration $i$ has type U (resp. D).
Table~\ref{tab:sigmpim} illustrates the possible values of $\sigm_j(b_i+\ell\eps)$ and $\pim_j(b_i+\ell\eps)$ for $0\le j\le k$ and $-1\le \ell\le k+1$. The letter $\delta$ is used to indicate that
some entries may be modified by $\pm1$, with some abuse of notation, since the value of $\delta$ is not necessarily the same for the different entries of the table. Each $\delta\in\{0,1\}$ if iteration $i$ has type U, and
$\delta\in\{0,-1\}$ if it has type D.
In any case, the leftmost entry of block $p$ either stays fixed or moves up (in type D) or down (in type U)
by 1 only once throughout iteration $i$, when it is switched with leftmost entry of $\Gamma_i$.
The $>$ and $<$ signs between the entries indicate their relative order
(with no sign when it could go in either direction). The symbol~$^\bullet$ between entries indicates that the relative order disagrees with that of the corresponding entries (in the same positions) in $\pi$, or equivalently in $\sigma$, by Proposition~\ref{prop:descents}. The symbol~$^\circ$ means that the relative order agrees, and the symbol~$^\odot$ means that it could agree or disagree.
For all the entries in the table to be defined, one has to assume that $b_i-\eps,b_i+(k{+}1)\eps\in[n]$. The table displays the case $\pim_0(b_i-\eps)=s_0$, which holds
unless $b_i-\eps$ is the rightmost entry of a cycle, in which case $\pim_0(b_i-\eps)>s_0$; in both cases, $|\pim_j(b_i-\eps)-\pim_{j-1}(b_i-\eps)|\le 1$ for $1\le j\le k$.
Similarly, $\pim_k(b_i+(k{+}1)\eps)=s_{k+1}$, as shown in the table, unless $b_i+(k{+}1)\eps$ is the rightmost entry of a cycle, in which case $\pim_k(b_i+(k{+}1)\eps)>s_{k+1}$; in both cases, $|\pim_j(b_i+(k{+}1)\eps)-\pim_{j-1}(b_i+(k{+}1)\eps)|\le 1$ for $1\le j\le k$.

\begin{table}[htb]
$$\begin{array}{c|c@{}c@{}c@{\,}l@{}c@{\,}l@{}c@{\,}l@{}c@{\,}l@{\,}c@{\,}l@{}c@{}l@{}c@{}c@{}c@{}}
& b_i-\eps && b_i &&b_i+\eps &&b_i+2\eps &&b_i+3\eps && \dots &&b_i+(k{-}1)\eps &&b_i+k\eps &&b_i+(k{+}1)\eps \\ \hline
\sigm_0 & s_0 & ^{\odot} & a_i^0 & <^\bullet & s_1 &>^\circ& s_2{+}\delta &>^\circ& s_3{+}\delta &>^\circ& \dots &>^\circ&  s_{k-1}{+}\delta &>^\circ& s_k{+}\delta &^\circ&  \\
\sigm_1 & s_0{\pm}\delta & ^\circ & s_1 & >^\circ & a_i^1 &<^\bullet& s_2 &>^\circ& s_3{+}\delta &>^\circ& \dots &>^\circ&  s_{k-1}{+}\delta &>^\circ& s_k{+}\delta &^\circ&  \\
\sigm_2 &  & ^\circ & s_1{-}\delta & >^\circ & s_2 &>^\circ& a_i^2 &<^\bullet& s_3 &>^\circ& \dots &>^\circ&  s_{k-1}{+}\delta &>^\circ& s_k{+}\delta &^\circ&  \\
\vdots & &&\vdots&&\vdots&&\vdots&&\vdots&& \ddots &&\vdots&&\vdots&&\\
\sigm_{k-1} &  & ^\circ & s_1{-}\delta & >^\circ & s_2{-}\delta &>^\circ& s_3{-}\delta &>^\circ& s_4{-}\delta &>^\circ& \dots &>^\circ&  a_i^{k-1} &<^\bullet& s_k &^\circ& s_{k+1}{\pm}\delta \\
\sigm_{k} &  & ^\circ & s_1{-}\delta & >^\circ & s_2{-}\delta &>^\circ& s_3{-}\delta &>^\circ& s_4{-}\delta &>^\circ& \dots &>^\circ& s_k &>^\circ& a_i^{k} &^\circ & s_{k+1} \\ \hline
\pim_0 & s_0 & ^\circ & a_{i+1} & >^\circ & s_1 &>^\circ& s_2{+}\delta &>^\circ& s_3{+}\delta &>^\circ& \dots &>^\circ&  s_{k-1}{+}\delta &>^\circ& s_k{+}\delta &^\circ&  \\
\pim_1 & s_0{\pm}\delta & ^\circ & s_1 & <^\bullet & a_{i+1} &>^\circ& s_2 &>^\circ& s_3{+}\delta &>^\circ& \dots &>^\circ&  s_{k-1}{+}\delta &>^\circ& s_k{+}\delta &^\circ&  \\
\pim_2 &  & ^\circ & s_1{-}\delta & >^\circ & s_2 &<^\bullet& a_{i+1} &>^\circ& s_3 &>^\circ& \dots &>^\circ&  s_{k-1}{+}\delta &>^\circ& s_k{+}\delta &^\circ&  \\
\vdots & &&\vdots&&\vdots&&\vdots&&\vdots&& \ddots &&\vdots&&\vdots&&\\
\pim_{k-1} &  & ^\circ & s_1{-}\delta & >^\circ & s_2{-}\delta &>^\circ& s_3{-}\delta &>^\circ& s_4{-}\delta &>^\circ& \dots &<^\bullet&  a_{i+1} &>^\circ& s_k &^\circ& s_{k+1}{\pm}\delta \\
\pim_{k} &  & ^\circ & s_1{-}\delta & >^\circ & s_2{-}\delta &>^\circ& s_3{-}\delta &>^\circ& s_4{-}\delta &>^\circ& \dots &>^\circ& s_k &<^\bullet& a_{i+1} & ^{\odot} & s_{k+1}
\end{array}
$$
\caption{The values $\sigm_j(b_i+\ell\eps)$ and $\pim_j(b_i+\ell\eps)$ for $0\le j\le k$ and $-1\le \ell\le k+1$. The relative order of adjacent entries is indicated,
as well as whether it agrees or not with the relative order of the corresponding entries in $\sigma$ and $\pi$.\label{tab:sigmpim}}
\end{table}

The values and the signs $>,<$ in Table~\ref{tab:sigmpim} are a consequence of Lemmas~\ref{obsa}-\ref{obsf}.
The symbols~$^\bullet,^\circ,^\odot$ are justified in the following lemma. We first introduce some notation in order to state it.
For $\tau\in\S_n$, let $W(\tau)$ be the set of (unordered) pairs $\{t,t+1\}$ such that $t\in (D(\pi)\cap[n-1])\bigtriangleup D(\tau)$, that is, pairs of adjacent positions
where the descents of $\pi$ (or equivalently $\sigma$, by Proposition~\ref{prop:descents} and $\tau$ disagree.
Let $$W_R=\bigcup_{i<\ell\le r-1}\{\{b_\ell-\eps,b_\ell\},\{b_\ell,b_\ell+\eps\}\},$$
and let $$W_L=\bigcup_{1\le\ell<i}\{\{b_\ell+(k_\ell{-}1)\eps,b_\ell+k_\ell\eps\},\{b_\ell+k_\ell\eps,b_\ell+(k_\ell{+}1)\eps\}\},$$
where $k_\ell$ is the number of subiterations of iteration $\ell$ of $\bij$ (note that $k_i=k$ by definition).

\begin{lemma}\label{lem:relorder}
For $1\le j\le k-1$, $$W(\sigm_j)\subseteq\{\{b_i+j\eps,b_i+(j{+}1)\eps\}\}\cup W_R,$$
$$W(\pim_j)\subseteq\{\{b_i+(j{-}1)\eps,b_i+j\eps\}\}\cup W_L.$$
\begin{eqnarray*}
&&W(\sigm_0)\subseteq\{\{b_i-\eps,b_i\},\{b_i,b_i+\eps\}\}\cup W_R, \qquad W(\sigm_k)\subseteq W_R, \\
&&W(\pim_0)\subseteq W_L, \qquad W(\pim_k)\subseteq\{\{b_i+(k{-}1)\eps,b_i+k\eps\},\{b_i+k\eps,b_i+(k{+}1)\eps\}\}\cup W_L.
\end{eqnarray*}
\end{lemma}

\begin{proof} This argument essentially replicates the proof of Proposition~\ref{prop:descents}, using the notation introduced above and keeping track of $W(\pim_j)$ as well.
The proof is by induction on $i$, so let us forget for a moment that $i$ is fixed in this section.
When $i=0$, $W(\pim_0)=\emptyset$ because $\pim_0=\pi$, and $W(\sigm_0)\subseteq\bigcup_{1\le\ell\le r-1}\{\{b_\ell-\eps,b_\ell\},\{b_\ell,b_\ell+\eps\}\}$, as in the proof of Proposition~\ref{prop:descents}.
For $i>0$, the statements about $W(\sigm_0)$ and $W(\pim_0)$ follow by the induction hypothesis and the fact that $\sigm_0$ and $\pim_0$ in iteration $i$ equal $\sigm_{k_{i-1}}$ and $\pim_{k_{i-1}}$ in iteration $i-1$, respectively.

Going back to our fixed value of $i$, we will prove the statements about $W(\sigm_j)$ and $W(\pim_j)$ for $1\le j\le k$ by looking at the switches performed during the $i$-th iteration.
This iteration starts by checking $P(b_i,b_i-1)$ and $P(b_i,b_i+1)$ to determine whether these pairs are in $W(\sigm_0)$. If either of them is,
subiteration 1 switches a segment of entries ending in $b_i$ with a segment of entries ending in $b_i+\eps$ (namely those in $R_1$), with $\eps$ chosen so that $\sigm_0(b_i+\eps)$ is largest.
The switch guarantees that neither of the pairs $\{b_i,b_i-1\},\{b_i,b_i+1\}$ is in $W(\sigm_1)$,
but it adds the pair $\{b_i,b_i+\eps\}$ to $W(\pim_1)$. The switches performed in step II of subiteration 1 prevent other pairs from being added to $W(\sigm_1)$ or, with the possible exception
of $\{b_i+\eps,b_i+2\eps\}$ and $\{b_i-\eps,b_i\}$, to $W(\pim_1)$. We now show that these pairs are not added to $W(\pim_1)$.
For the pair $\{b_i+\eps,b_i+2\eps\}$, we have that $a_{i+1}=\pim_1(b_i+\eps)>\pim_1(b_i+2\eps)=s_2$, which are in the same relative order as $\pi(b_i+\eps)>\pi(b_i+2\eps)$, by Lemma~\ref{obsa},
so $\{b_i+\eps,b_i+2\eps\}\notin W(\pim_1)$.
For the pair $\{b_i-\eps,b_i\}$ (assuming $b_i-\eps\in[n]$), note that $\pim_0(b_i-\eps)=s_0$ unless $\pim_0(b_i-\eps)$ is the first entry of a cycle,
so it is never the case that $s_1<\pim_0(b_i-\eps)<a_{i+1}$, using Lemma~\ref{lem:typeUD}(\ref{s0UD}) below.
But since $\pim_0(b_i)=a_{i+1}$ and $\pim_1(b_i)=s_1$,
the relative order of $\pim_j(b_i-\eps)$ and $\pim_j(b_i)$ is the same for $j=0$ and $j=1$, so $\{b_i-\eps,b_i\}\notin W(\pim_1)$.

After subiteration $1$, the relative order of $\sigm_1(b_i+\eps)$ and $\sigm_1(b_i+2\eps)$ is different from the relative order of $\pi(b_i+\eps)$ and $\pi(b_i+2\eps)$, unless $k=1$.
The condition $P(b_i+\eps,b_i+2\eps)$ checks that this is the case, and then step I of subiteration $2$
switches $b_i+\eps$ and $b_i+2\eps$ to fix the problem, with step II preventing other pairs from being added to $W(\sigm_2)$. The switch between $b_i+\eps$ and $b_i+2\eps$ removes
$\{b_i,b_i+\eps\}$ from $W(\pim_2)$ but adds $\{b_i+\eps,b_i+2\eps\}$ to it. As before, the pair $\{b_i+2\eps,b_i+3\eps\}$ is not added to $W(\pim_2)$ because
$a_{i+1}=\pim_2(b_i+2\eps)>\pim_2(b_i+3\eps)=s_3$, which are in the same relative order as $\pi(b_i+2\eps)>\pi(b_i+3\eps)$. Step II again prevents any other pairs from being added to $W(\pim_2)$.

Subiterations from $3$ to $k$ proceed analogously. At the end of subiteration $k$, $P(b_i+k\eps,b_i+(k{+}1)\eps)$ is false, which means that either $b_i+(k{+}1)\eps\in\{0,n+1\}$
or the relative order of $\sigm_k(b_i+k\eps)$ and $\sigm_k(b_i+(k{+}1)\eps)$ agrees with the relative order of $\pi(b_i+k\eps)$ and $\pi(b_i+(k{+}1)\eps)$, so $\{b_i+k\eps,b_i+(k{+}1)\eps\}\notin W(\sigm_k)$.
On the other hand, subiteration $k$ has added $\{b_i+(k{-}1)\eps,b_i+k\eps\}$ to $W(\pim_k)$, and it is possible that $\{b_i+k\eps,b_i+(k{+}1)\eps\}\in W(\pim_k)$ as well.
\end{proof}

Now we can prove the analogue of Lemma~\ref{lem:typeH1} when iteration $i$ has arbitrary type. Note that even though it is no longer true in general that $s_j>a_i$, Lemma~\ref{obsa} guarantees that $s_j>a_i^{j-1}$
for $1\le j\le k$.

\begin{lemma}\label{lem:typeUD}
In iteration $i$ we have that
\ben \item \label{eq:orderaiUD}
$a_{i+1}>s_1>s_2>\dots>s_k$,
\item\label{s0UD} $s_0$ does not satisfy $a_{i+1}>s_0>s_1$,
\item\label{skUD} $s_{k+1}$ does not satisfy $s_k>s_{k+1}>a_i^k$.
\een
\end{lemma}

\begin{proof}
By Lemma~\ref{lem:relorder}, we have $\pim_0(b_i)>\pim_0(b_i+\eps)>\pim_0(b_i+2\eps)>\dots>\pim_0(b_i+k\eps)$. This is equivalent to $a_{i+1}>s_1>s_2+\delta>\dots>s_k+\delta$, as shown in Table~\ref{tab:sigmpim},
with each $\delta\in\{0,1\}$ if iteration $i$ has type U, $\delta\in\{0,-1\}$ if it has type D, and $\delta=0$ if it has type H. Since by definition the values $s_1,s_2,\dots,s_k$ are all different, part~(\ref{eq:orderaiUD}) follows.

Part~(\ref{s0UD}) is proved in the same way as part~(\ref{s0}) of Lemma~\ref{lem:typeH1}. Indeed, the argument was independent of the type of the $i$-th iteration, using that $\sigm_0(b_i)=a_i^0=a_i<s_1$.

For part~(\ref{skUD}), assume that $s_{k+1}=\sigm_k(b_i+(k{+}1)\eps)$ is defined and that $s_k>s_{k+1}>a_i^k$.
The condition $P(b_i+k\eps,b_i+(k{+}1)\eps)$ is that $$\pi(b_i+k\eps)>\pi(b_i+(k{+}1)\eps) \mbox{ and } \sigm_k(b_i+k\eps)<\pi(b_i+(k{+}1)\eps).$$
The second inequality is by definition equivalent to $a_i^k<s_{k+1}$; the first one, by Lemma~\ref{lem:relorder} and Table~\ref{tab:sigmpim}, is equivalent to $s_k=\pi_{k-1}(b_i+k\eps)>\pi_{k-1}(b_i+(k{+}1)\eps)=s_{k+1}\pm\delta$,
using the fact that $b_i+(k{+}1)\eps$ is not the rightmost entry of a cycle in $\pim_{k}$, otherwise $s_{k+1}\ge a_{i+1}>s_k$.
Since $s_k\neq s_{k+1}$, the condition $s_k>s_{k+1}>a_i^k$ implies that $P(b_i+k\eps,b_i+(k{+}1)\eps)$ holds, but then the algorithm would have switched $b_i+k\eps$ with $b_i+(k{+}1)\eps$
instead of ending iteration $i$ right after subiteration~$k$.
\end{proof}

We can now conclude the proof of the following statement.

\begin{prop}\label{prop:inverses}
The maps $\bij$ and $\inv$ are inverses of each other.
\end{prop}

\begin{proof}
We have shown in Lemma~\ref{lem:bijinvH} that if the $i$-th iteration of $\bij$ has type H, then the $i$-th iteration of $\inv$ reverses the switches done by the $i$-th iteration of $\bij$.
It is enough to show that this is true as well when the $i$-th iteration of $\bij$ has type U or D.

We first show that for $1\le j\le k$, the fact that $P(b_i+(j{-}1)\eps,b_i+j\eps)$ holds on $\sigm_{j-1}$ implies that $Q(b_i+j\eps,b_i+(j-1)\eps)$ holds on $\pim_j$.
Indeed, $P(b_i+(j{-}1)\eps,b_i+j\eps)$ is the condition \beq\label{eq:Pbj}\pi(b_i+(j{-}1)\eps)>\pi(b_i+j\eps) \mbox{ and } \sigm_{j-1}(b_i+(j{-}1)\eps)<\sigm_{j-1}(b_i+j\eps).\eeq
By Lemma~\ref{lem:relorder}, the relative order of $\pi(b_i+(j{-}1)\eps)$ and $\pi(b_i+j\eps)$ is the same as that of $\pim_{j-1}(b_i+(j{-}1)\eps)$ and $\pim_{j-1}(b_i+j\eps)$, so
the first inequality in~(\ref{eq:Pbj}) is equivalent to $\pim_{j-1}(b_i+(j{-}1)\eps)>\pim_{j-1}(b_i+j\eps)$. But $\pim_{j-1}(b_i+(j{-}1)\eps)=a_{i+1}=\pim_j(b_i+j\eps)$, and $\pim_{j-1}(b_i+j\eps)=s_j=\pim_j(b_i+(j{-}1)\eps)$,
so we can write the inequality as $\pim_j(b_i+j\eps)>\pim_j(b_i+(j{-}1)\eps)$.
On the other hand, we have $\sigm_{j-1}(b_i+(j{-}1)\eps)=a_i^{j-1}=\sigm_{j}(b_i+j\eps)\pm1$ and $\sigm_{j-1}(b_i+j\eps)=s_j=\sigm_{j}(b_i+(j{-}1)\eps)$, so the second inequality in~(\ref{eq:Pbj})
implies that $\sigm_{j}(b_i+j\eps)<\sigm_{j}(b_i+(j{-}1)\eps)$, since these two quantities are never equal. This inequality is in turn equivalent to $\sigma(b_i+j\eps)<\sigma(b_i+(j{-}1)\eps)$
by Lemma~\ref{lem:relorder}. Thus, we have shown that the inequalities in~(\ref{eq:Pbj}) imply that
$$\pim_j(b_i+j\eps)>\pim_j(b_i+(j{-}1)\eps) \mbox{ and } \sigma(b_i+j\eps)<\sigma(b_i+(j{-}1)\eps),$$
which is by definition the condition $Q(b_i+j\eps,b_i+(j{-}1)\eps)$ on $\pim_j$. Alternatively, we could have argued that $Q(b_i+j\eps,b_i+(j{-}1)\eps)$ is equivalent to $a_{i+1}>s_j>a_i^{j-1}$, which holds by Lemmas~\ref{lem:typeUD}(\ref{eq:orderaiUD}) and~\ref{obsa}.

As in the proof of Lemma~\ref{lem:bijinvH}, suppose that $\pim=\pim_k$ right before the $i$-th iteration of~$\inv$. The $i$-th block of $\pim_k$ is $c_i,\dots,d_i$,
where $c_i=a_i^k$ and $d_i=b_i+k\eps$.
The $i$-th iteration of $\inv$ starts by checking whether $Q(b_i+k\eps,b_i+(k{-}1)\eps)$ and $Q(b_i+k\eps,b_i+(k{+}1)\eps)$ hold. We have seen in the above paragraph that $Q(b_i+k\eps,b_i+(k{-}1)\eps)$ holds.
The only situation that would prevent $b_i+k\eps$ from being switched with $b_i+(k{-}1)\eps$ by $\inv$ would be if $Q(b_i+k\eps,b_i+(k{+}1)\eps)$ held and $\pim_k(b_i+(k{+}1)\eps)<\pim_k(b_i+(k{-}1)\eps)$.
Let us assume for contradiction that this is the case. Then, by Lemma~\ref{obsf}, $\pim_k(b_i+(k{-}1)\eps)=\sigm_k(b_i+(k{-}1)\eps)=s_k$, and by definition, $\pim_k(b_i+(k{+}1)\eps)\ge \sigm_k(b_i+(k{+}1)\eps)=s_{k+1}$,
so $s_{k+1}<s_k$ in this case.
Besides, for $Q(b_i+k\eps,b_i+(k{+}1)\eps)$ to hold, $\sigma(b_i+k\eps)<\sigma(b_i+(k{+}1)\eps)$, which by Lemma~\ref{lem:relorder} is equivalent to $a_i^k=\sigm_k(b_i+k\eps)<\sigm_k(b_i+(k{+}1)\eps)=s_{k+1}$. Putting these two statements
statements together, $a_i^k<s_{k+1}<s_k$, contradicting Lemma~\ref{lem:typeUD}(\ref{skUD}).
Consequently, the $i$-th iteration of $\inv$ starts by switching $b_i+k\eps$ and $b_i+(k{-}1)\eps$ and possibly some entries preceding them, undoing the switches preformed by subiteration $k$ of $\bij$ and obtaining $\pim_{k-1}$.

Next, for $j=k{-}1,k{-}2,\dots,1$, the condition $Q(b_i+j\eps,b_i+(j{-}1)\eps)$ holds in $\pim_j$ as shown above, so $\inv$ switches $b_i+j\eps$ and $b_i+(j{-}1)\eps$ along with the necessary entries preceding them, undoing
the switches preformed by subiteration $j$ of $\bij$ and recovering $\pim_{j-1}$.

When $\pim_0$ is reached, $\inv$ checks condition $Q(b_i,b_i-\eps)$, that is, whether \beq\label{eq:pim0UD}\pim_0(b_i)>\pim_0(b_i-\eps)\ \textrm{ and }\ \sigma(b_i)<\sigma(b_i-\eps),\eeq assuming that $b_i-\eps\notin\{0,n+1\}$.
As in the proof of Lemma~\ref{lem:bijinvH}, $\pim_0(b_i)=a_{i+1}$ and $\pim_0(b_i-\eps)\ge\sigm_0(b_i-\eps)=s_0$.
By Lemma~\ref{lem:relorder}, the second inequality in~(\ref{eq:pim0UD})
is equivalent to $s_1=\sigm_1(b_i)<\sigm_1(b_i-\eps)=s_0\pm\delta$, which implies that $s_1<s_0$ since by definition $s_1\neq s_0$.
But if both inequalities~(\ref{eq:pim0UD}) held, then $a_{i+1}>s_0>s_1$, contradicting Lemma~\ref{lem:typeUD}(\ref{s0UD}). Thus, iteration $i$ of $\inv$ stops here.

We have shown that for each $i=r{-}1,r{-}2,\dots,1$, the $i$-th iteration of $\inv$ undoes the switches performed by the $i$-th iteration of $\bij$.
This proves that $\inv(\bij(\pi))=\pi$ for all $\pi\in\C_{n+1}$, and since $|\C_{n+1}|=|\S_n|=n!$, it follows that $\inv=\bij^{-1}$.
\end{proof}

Propositions~\ref{prop:descents} and~\ref{prop:inverses} together complete the proof of Theorem~\ref{thm:bij}.

\section{Consequences}\label{sec:consequences}

The following is an obvious consequence of Theorem~\ref{thm:bij}. We state it separately in order to refer to it later.

\begin{corollary}\label{cor:cycles} For every $n$ and every $I\subseteq[n-1]$,
$$|\{\pi\in\C_{n+1}\,:\,D(\pi)\cap[n-1]=I\}|=|\{\sigma\in\S_n\,:\,D(\sigma)=I\}|.$$
\end{corollary}

This result has the following probabilistic interpretation. Choose a permutation $\pi\in\S_{n+1}$ uniformly at random. Then, for any given $I\subseteq[n-1]$, the event that $D(\pi)\cap[n-1]=I$
and the event that $\pi$ is a cyclic permutation are independent. To see this, note that the relative order of $\pi(1)\pi(2)\dots\pi(n)$ is given by a uniformly random permutation in $\S_n$. Thus,
for any fixed $I\subseteq[n-1]$, the probability that $D(\pi)\cap[n-1]=I$ for a random $\pi\in\S_{n+1}$ is the same as the probability that $D(\sigma)=I$ for a random $\sigma\in\S_n$, which by Corollary~\ref{cor:cycles}
is the same as the probability that $D(\pi)\cap[n-1]=I$ for a random $\pi\in\C_{n+1}$.

Our next goal is to show that Conjecture~\ref{conj:Eli} follows from Theorem~\ref{thm:bij}.
First, instead of the set $\T^0_n$, it will be more convenient for the sake of notation to consider the set $\U_n$
consisting of $n$-cycles in one-line notation in which one entry has been replaced with $n+1$. For example, $\U_3=\{431,241,234,412,342,314\}$.

\begin{corollary}\label{cor:biju} For every $n$ there is a bijection $\biju$ between $\U_{n}$ and $\S_n$ such that if $\tau\in\U_n$ and $\sigma=\biju(\tau)$, then
$$D(\tau)=D(\sigma).$$
Additionally, if $n+1$ is in position $k$ of $\tau$, then $\sigma(k)=n$.
\end{corollary}

\begin{proof}
Let $\tau\in\U_n$ and suppose it has been obtained from an $n$-cycle $\pi$ by replacing $\pi(k)$ with $n+1$ in the one-line notation. Write $\pi$ in cycle form with $k$ at the end,
say $\pi=(t_1,t_2,\dots,t_{n-1},k)$, and let $\pi'=(t_1,t_2,\dots,t_{n-1},k,n+1)\in\C_{n+1}$. Clearly, $D(\tau)=D(\pi')\cap[n-1]$, and the map $\tau\mapsto\pi'$ is a bijection between $\U_n$ and $\C_{n+1}$.
Let $\sigma=\bij(\pi')$. By Theorem~\ref{thm:bij}, $D(\pi')\cap[n-1]=D(\sigma)$, and by Proposition~\ref{prop:properties}(\ref{n}), $\sigma(k)=n$.
\end{proof}

The following corollary proves Conjecture~\ref{conj:Eli}.
\begin{corollary}\label{cor:Elishift} For every $n$ there is a bijection $\biju'$ between $\T^0_{n}$ and $\S_n$ such that if $\tau\in\T^0_n$ and $\sigma=\biju'(\tau)$, then
$$D(\tau)=D(\sigma).$$
Additionally, if $0$ is in position $k$ of $\tau$, then $\sigma(k)=1$.
\end{corollary}

\begin{proof}
Given $\tau\in\T^0_n$ obtained from an $n$-cycle $\pi$ by replacing $\pi(k)$ with $0$ in its one-line notation, let $\wh\tau\in\U_n$ be obtained from $\wh\pi$ (see the definition in the introduction) by replacing $\wh\pi(n+1-k)$ with $n+1$.
It is clear that for $1\le i\le n-1$, $i\in D(\wh\tau)$ if and only if $n+1-i\notin D(\tau)$. Let $\sigma=\biju'(\tau)=\wh{\biju(\wh\tau)}$. Then, for $1\le i\le n-1$,
$$i\in D(\sigma) \ \Leftrightarrow\  n+1-i\notin D(\biju(\wh\tau))=D(\wh\tau) \ \Leftrightarrow\ i\in D(\tau).$$ Also $\wh\sigma(n+1-k)=n$, so $\sigma(k)=1$.
\end{proof}

The final result of this section can be seen as a generalization of Corollary~\ref{cor:cycles}. We give a bijective proof of it.

\begin{corollary}\label{cor:cyclesu} Fix $1\le m\le n$ and let $J=[n-1]\setminus\{m-1,m\}$. For every $I\subseteq J$,
$$|\{\pi\in\C_{n}\,:\,D(\pi)\cap J=I\}|=|\{\sigma\in\S_{n}\,:\,\sigma(m)=1,\,D(\sigma)\cap J=I\}|.$$
\end{corollary}

\begin{proof}
Let $\pi\in\C_n$ with $D(\pi)\cap J=I$. Let $\tau\in\T^0_n$ be obtained by replacing $\pi(m)$ with $0$ in the one-line notation of $\pi$, and let $\sigma=\biju'(\tau)$. By Corollary~\ref{cor:Elishift}, $\sigma(m)=1$
and $D(\sigma)\cap J=D(\tau)\cap J=D(\pi)\cap J=I$.
\end{proof}

\section{Related work and non-bijective proofs}\label{sec:related}

In this section we introduce some related work of Gessel and Reutenauer~\cite{GR}, which will allow us to give
non-bijective proofs of Corollaries~\ref{cor:cycles} and~\ref{cor:cyclesu}.
We start with some definitions. Let $X=\{x_1,x_2,\dots\}_<$ be a linearly ordered alphabet.
A {\em necklace} of length $\ell$ is a circular arrangement of $\ell$ beads which are labeled with elements of $X$. Two necklaces are considered the same if they are cyclic rotations of one another
(note that we do not allow reflections).
The cycle structure of a multiset of necklaces is the partition whose parts are the lengths of the necklaces in the multiset.
The {\em evaluation} of a multiset of necklaces is the monomial $x_1^{e_1}x_2^{e_2}\dots$ where $e_i$ is the number of beads with label $x_i$.

The following result is equivalent to Corollary~2.2 from~\cite{GR}.

\begin{theorem}[\cite{GR}]\label{thm:GR}
Let $I=\{i_1,i_2,\dots,i_k\}_<\subseteq[n-1]$ and let $\lambda$ be a partition of~$n$. Then the number of permutations with cycle structure $\lambda$ and descent set contained in $I$ equals the number of multisets of necklaces with cycle structure $\lambda$ and evaluation $x_1^{i_1}x_2^{i_2-i_1}\dots x_k^{i_k-i_{k-1}}x_{k+1}^{n-i_k}$ .
\end{theorem}

We can now give direct, non-bijective proofs of Corollaries~\ref{cor:cycles} and~\ref{cor:cyclesu}.

\begin{proof}[Alternate proof of Corollary~\ref{cor:cycles}]
Suppose that $I=\{i_1,i_2,\dots,i_k\}_<$, and let $I'=I\cup\{n\}$.
By Theorem~\ref{thm:GR}, the number of permutations $\pi\in\C_{n+1}$ with $D(\pi)\subseteq I'$ (equivalently, $D(\pi)\cap[n-1]\subseteq I$) equals the number of necklaces with evaluation
$$x_1^{i_1}x_2^{i_2-i_1}\dots x_k^{i_k-i_{k-1}}x_{k+1}^{n-i_k}x_{k+2}.$$
By first choosing the bead labeled $x_{k+2}$, it is clear that the number of such necklaces is $$\binom{n}{i_1,i_2-i_1,\dots,i_k-i_{k-1},n-i_k}.$$
But this is precisely (see~\cite{EC1}) the number of permutations in $\S_n$ whose descent set is contained in $I$. Thus, we have shown that
$$|\{\pi\in\C_{n+1}\,:\,D(\pi)\cap[n-1]\subseteq I\}|=|\{\sigma\in\S_n\,:\,D(\sigma)\subseteq I\}|.$$
Since this holds for all $I\subseteq[n-1]$, the statement now follows by inclusion-exclusion:
\bmt |\{\pi\in\C_{n+1}\,:\,D(\pi)\cap[n-1]=I\}|=\sum_{J\subseteq I}(-1)^{|I|-|J|}|\{\pi\in\C_{n+1}\,:\,D(\pi)\cap[n-1]\subseteq J\}|\\
=\sum_{J\subseteq I}(-1)^{|I|-|J|}|\{\sigma\in\S_n\,:\,D(\sigma)\subseteq J\}|=|\{\sigma\in\S_n\,:\,D(\sigma)=I\}|.
\end{multline*}
\end{proof}

Note that even though a bijective proof of Theorem~\ref{thm:GR} is implicit in~\cite{GR}, the last inclusion-exclusion step in the above proof of Corollary~\ref{cor:cycles} makes it non-bijective.

\begin{proof}[Alternate proof of Corollary~\ref{cor:cyclesu}]
Let $I=\{i_1,i_2,\dots,i_k\}_<$. Assume first that $1<m<n$, and let $$I'=I\cup\{m-1,m\}=\{i_1,i_2,\dots,i_j,m-1,m,i_{j+1},\dots,i_k\}_<.$$
By Theorem~\ref{thm:GR}, the number of permutations $\pi\in\C_{n}$ with $D(\pi)\subseteq I'$ (equivalently, $D(\pi)\cap J\subseteq I$) equals the number of necklaces with evaluation
$$x_1^{i_1}x_2^{i_2-i_1}\dots x_j^{i_j-i_{j-1}}x_{j+1}^{m-1-i_j}x_{j+2}x_{j+3}^{i_{j+1}-m}\dots x_{k+2}^{i_k-i_{k-1}}x_{k+3}^{n-i_k}.$$
By first choosing the bead labeled $x_{j+2}$, it is clear that the number of such necklaces is
$$\binom{n-1}{i_1,i_2-i_1,\dots,i_j-i_{j-1},m-1-i_j,i_{j+1}-m,\dots,i_k-i_{k-1},n-i_k}.$$
But this is precisely the number of permutations $\sigma\in\S_n$ with $\sigma(m)=1$ whose descent set satisfies $D(\sigma)\cap J\subseteq I$.
Indeed, each partition of $\{2,3,\dots,n\}$ into blocks of sizes $$i_1,i_2-i_1,\dots,i_j-i_{j-1},m-1-i_j,i_{j+1}-m,\dots,i_k-i_{k-1},n-i_k$$
corresponds to the permutation whose first $i_1$ entries are the elements of the first block in increasing order, followed by the $i_2-i_1$ elements
of the second block in increasing order, until we get to the $m$-th entry, which is $1$, after which the $i_{j+1}-m$ elements of the $(j{+}1)$-st block follow in increasing order, and so on.
This proves that
$$|\{\pi\in\C_{n}\,:\,D(\pi)\cap J\subseteq I\}|=|\{\sigma\in\S_n\,:\,\sigma(m)=1,\,D(\sigma)\cap J\subseteq I\}|.$$
As before, since this equality holds for all $I\subseteq J$, the main statement now follows by inclusion-exclusion.

If $m=1$ or $m=n$, we let $I'=I\cup\{m\}=\{1,i_1,i_2,\dots,i_k\}$ or $I'=I\cup\{m-1\}=\{i_1,i_2,\dots,i_k,m-1\}$, respectively, and apply an analogous argument.
\end{proof}

\bs

We end this section with another application of the work of Gessel and Reutenauer. We show that \cite[Lemma 3.4]{GR} can be used to provide an explicit bijection
that solves a generalization of Problem~\ref{prob:EFW}. Indeed, since it preserves the cycle structure, the following bijection sends derangements to derangements.

\begin{prop}\label{prop:subsets} For any two subsets $I,J\subseteq[n-1]$ with the same associated partition, there exists a bijection between
$\{\pi\in\S_n\,:\,D(\pi)\subseteq I\}$ and $\{\sigma\in\S_n\,:\,D(\sigma)\subseteq J\}$ that preserves the cycle structure.
\end{prop}

\begin{proof}
Let $\pi\in\S_n$ with $D(\pi)\subseteq I$, where $I=\{i_1,i_2,\dots,i_k\}_<$, and let $\lambda$ be the cycle structure of $\pi$.
For convenience, define $i_0=0$ and $i_{k+1}=n$, and let $$(r_1,r_2,\dots,r_{k+1})=(i_1,i_2-i_1,\dots,i_k-i_{k-1},n-i_k)$$ be the corresponding composition on $n$.
Similarly, let $(s_1,s_2,\dots,s_{k+1})$ be the composition of $n$ corresponding to $J$. Since the associated partitions are the same, there is a permutation $\alpha$ of the indices such that
$r_j=s_{\alpha(j)}$ for $1\le j\le k+1$.

Write $\pi$ as a product of cycles and for each $1\le j\le k+1$, replace the entries $i_{j-1}+1,i_{j-1}+2,\dots,i_j$ with $x_{\alpha(j)}$,
thus obtaining a multiset of necklaces. For each bead, consider the periodic sequence obtained by reading the necklace starting at that bead. Now, order these sequences lexicographically (if there are repeated necklaces, first
choose an order among them), and label the vertices with $1,2,\dots,n$ according to this order. This yields the cycle form of a permutation $\sigma$, which clearly has cycle structure $\lambda$.
It follows from~\cite{GR} that $D(\sigma)\subseteq J$, and that the map $\pi\mapsto\sigma$ is a bijection. In fact, this map essentially amounts to applying the bijection $U$ from
\cite[Lemma 3.4]{GR} to a word whose standard permutation is $\pi^{-1}$, then replacing each $x_j$ with $x_{\alpha(j)}$ in the necklaces, and finally applying the inverse of $U$.
\end{proof}

\begin{example}
Let $n=12$, $I=\{2,8\}$ and $J=\{4,6\}$, so $(r_1,r_2,r_3)=(2,6,4)$ and $(s_1,s_2,s_3)=(4,2,6)$. Let $$\pi=3\:4\:1\:2\:5\:9\:11\:12\:6\:7\:8\:10=(1,3)(2,4)(5)(6,9)(7,11,8,12,10),$$
with $D(\pi)=\{2,8\}=I$.
After replacing $1,2$ with $x_{\alpha(1)}=x_2$, $3,4,5,6,7,8$ with $x_{\alpha(2)}=x_3$, and $9,10,11,12$ with $x_{\alpha(3)}=x_1$,
we obtain the multiset of necklaces $$(x_2,x_3)(x_2,x_3)(x_3)(x_3,x_1)(x_3,x_1,x_3,x_1,x_1).$$ The corresponding periodic sequences are
\bmt(x_2x_3x_2x_3\dots,x_3x_2x_3x_2\dots)(x_2x_3x_2x_3\dots,x_3x_2x_3x_2\dots)\\
(x_3x_3\dots)(x_3x_1x_3x_1\dots,x_1x_3x_1x_3\dots)\\
(x_3x_1x_3x_1x_1\dots,x_1x_3x_1x_1x_3\dots,x_3x_1x_1x_3x_1\dots,x_1x_1x_3x_1x_3\dots,x_1x_3x_1x_3x_1\dots),
\end{multline*}
and ordering them lexicographically we obtain the permutation $$\sigma=(5,10)(6,11)(12)(9,4)(8,2,7,1,3)=3\:7\:8\:9\:10\:11\:1\:2\:4\:5\:6\:12,$$ with $D(\sigma)=\{6\}\subseteq J$.

If instead we had had $J'=\{2,6\}$, with composition $(2,4,6)$, the permutation corresponding to $\pi$ would have been
$$\sigma'=(1,7)(2,8)(12)(11,6)(10,4,9,3,5)=7\:8\:5\:9\:10\:11\:1\:2\:3\:4\:6\:12,$$ with $D(\sigma')=\{2,6\}=J'$.
\end{example}

\section*{Acknowledgement}
The author thanks an anonymous referee for useful suggestions.

\end{document}